\documentclass{amsart}

\usepackage{amsmath,amssymb,color}

\usepackage[latin1]{inputenc}

\newtheorem{theorem}{Theorem}[section]
\newtheorem{lemma}[theorem]{Lemma}
\newtheorem{proposition}[theorem]{Proposition}
\newtheorem{corollary}[theorem]{Corollary}
\newtheorem{remark}[theorem]{Remark}

\newcommand{\mc}[1]{{\mathcal #1}}
\newcommand{\mf}[1]{{\mathfrak #1}}

\newcommand{\bb}[1]{{\mathbb #1}}
\newcommand{\bs}[1]{{\boldsymbol #1}}

\newcommand{\<}{\langle}
\renewcommand{\>}{\rangle}
\newcommand{\x}{\!\otimes}

\begin{document}

\title[$W$-Sobolev spaces]{$W$-Sobolev spaces: Theory, Homogenization and Applications}

\author [A.B. Simas] {Alexandre B. Simas}
\author [F.J. Valentim]{Fábio J. Valentim}
\address{\hfill\break\indent
IMPA \hfill\break\indent
Estrada Dona Castorina 110, \hfill\break\indent
J. Botanico, 22460 Rio
de Janeiro, Brazil\\
Tel: +55 21 2529 5215}
\email{alesimas@impa.br\and valentim@impa.br}

\thanks{Research supported by CNPq}

\noindent


\begin{abstract}
Fix strictly increasing right continuous functions with left limits $W_i:\bb R \to \bb R$, $i=1,\ldots,d$, and let $W(x) = \sum_{i=1}^d W_i(x_i)$ for $x\in\bb R^d$. We construct the $W$-Sobolev spaces, which consist of functions $f$ having weak generalized gradients $\nabla_W f = (\partial_{W_1} f,\ldots,\partial_{W_d} f)$. Several properties, that are analogous to classical results on Sobolev spaces, are obtained. $W$-generalized elliptic and parabolic equations are also established, along with results on existence and uniqueness of weak solutions of such equations. Homogenization results of suitable random operators are investigated. Finally, as an application of all the theory developed, we prove a hydrodynamic limit for gradient processes with conductances (induced by $W$) in random environments.
\end{abstract}

\maketitle

\section{Introduction}
\label{sec1}
The space of functions that admit differentiation in a weak sense has been widely studied in the mathematical literature. The usage of such spaces provides a wide application to the theory of partial differential equations (PDE), and to many other areas of pure and applied mathematics.
These spaces have become associated with the name of the late Russian mathematician S. L. Sobolev, although their origins predate his major contributions to their development in the late 1930s. In theory of PDEs, the idea of Sobolev space allows one to introduce the notion of weak solutions whose existence, uniqueness, regularities, and well-posedness are based on tools of functional analysis. 

In classical theory of PDEs, two important classes of equations are: elliptic and parabolic PDEs. They are second-order PDEs, with some constraints (coerciveness) in the higher-order terms. The elliptic equations typically model the flow of some chemical quantity within some region, whereas the parabolic equations model the time evolution of such quantities. Consider the following particular classes of elliptic and parabolic equations:
\begin{equation}\label{elipara}
\sum_{i=1}^d \partial_{x_i}\partial_{x_i} u(x) = g(x),\quad\hbox{and}\quad \left\{\begin{array}{c}
\partial_tu(t,x) = \sum_{i=1}^d \partial_{x_i}\partial_{x_i} u(t,x),\\
u(0,x) = g(x),
\end{array}\right.
\end{equation}
for $t\in(0,T]$ and $x\in D$, where $D$ is some suitable domain, and $g$ is a function. Sobolev spaces are the natural environment to treat equations like \eqref{elipara} - an elegant exposition of this fact can be found in \cite{E}.

Consider the following generalization of the above equations:
\begin{equation}\label{elipara2}
\sum_{i=1}^d \partial_{x_i}\partial_{W_i} u(x) = g(x),\quad\hbox{and}\quad \left\{\begin{array}{c}
\partial_tu(t,x) = \sum_{i=1}^d \partial_{x_i}\partial_{W_i} u(t,x),\\
u(0,x) = g(x),
\end{array}\right.
\end{equation}
where $\partial_{W_i}$ stands for the generalized derivative operator, where, for each $i$, $W_i$ is a one-dimensional strictly increasing (not necessarily continuous) function. Note that if $W_i(x_i) = x_i$, we obtain the equations in \eqref{elipara}. This notion of generalized derivative has been studied by several authors in the literature, see for instance, \cite{dyn2,feller,lo1,m,mckean}. We also call attention to \cite{dyn2} since it provides a detailed study of such notion.  The equations in \eqref{elipara2} have the same physical interpretation as the equations in \eqref{elipara}. However, the latter covers more general situations. For instance, \cite{TC} and \cite{v} argue that these equations may be used to model a diffusion of particles within a region with membranes induced by the discontinuities of the functions $W_i$. Unfortunately, the standard Sobolev spaces are not suitable for being used as the space of weak solutions of equations in the form of \eqref{elipara2}.

One of our goals in this work is to define and obtain some properties of a space, which we call $W$-\emph{Sobolev space}. This space lets us formalize a notion of weak generalized derivative in such a way that, if a function is $W$-differentiable in the strong sense, it will also be differentiable in the weak sense, with their derivatives coinciding. Moreover, the $W$-Sobolev space will coincide with the standard Sobolev space if $W_i(x_i) = x_i$ for all $i$. With this in mind, we will be able to define weak solutions of equations in \eqref{elipara2}. We will prove that there exist weak solutions for such equations, and also, for some cases, the uniqueness of such weak solutions. Some analogous to classical results of Sobolev spaces are obtained, such as Poincaré's inequality and Rellich-Kondrachov's compactness theorem.

Besides the treatment of elliptic and parabolic equations in terms of these $W$-Sobolev spaces, we are also interested in studying \textit{Homogenization} and \textit{Hydrodynamic Limits}. The study of homogenization is motivated by several applications in mechanics, physics, chemistry and engineering. For example, when one studies the thermal or electric conductivity in heterogeneous materials, the macroscopic properties of crystals or the structure of polymers, are typically described in terms of linear or non-linear PDEs for medium with periodic or quasi-periodic structure, or, more generally, stochastic. 

We will consider stochastic homogenization. In the stochastic context, several works on homogenization of operators with random coefficients have been published (see, for instance, \cite{papa,pr} and references therein). In homogenization theory, only the stationarity of such random field is used. The notion of stationary random field is formulated in such a manner that it covers many objects of non-probabilistic nature, e.g., operators with periodic or quasi-periodic coefficients.

The focus of our approach is to study the asymptotic behavior of effective coefficients for a family of random difference schemes, whose coefficients can be obtained by the discretization of random high-contrast lattice structures. In this sense, we want to extend the theory of homogenization of random operators developed in \cite{pr}, as well as to prove its main Theorem (Theorem 2.16) to the context in which we have weak generalized derivatives.

Lastly, as an application of all the theory developed for $W$-Sobolev spaces, elliptic operators, parabolic equations and homogenization, we prove a hydrodynamic limit for \textit{gradient processes with conductances in random environments}. Hydrodynamic limit for gradient processes with conductances have been obtained in \cite{TC} for the one-dimensional setup and in \cite{v} for the $d$-dimensional setup. However, with the tools developed in our present article, the proof of the hydrodynamic limit on a more general setup (in random environments) turns out to be simpler and much more natural. Furthermore, the proof of this hydrodynamic limit also provides an existence theorem for the generalized parabolic equations such as the one in \eqref{elipara2}.

The hydrodynamic limit allows one to obtain a description of the thermodynamic characteristics (e.g., temperature, density, pressure, etc.) of infinite systems assuming that the underlying dynamics is stochastic and follows the statistical mechanics approach introduced by Boltzmann. More precisely, it allows one to deduce the macroscopic behavior of the system from the microscopic interaction among particles. We will consider a microscopic dynamics consisting of random walks on the lattice submitted to some local interaction, the so-called interacting particle systems introduced by Spitzer \cite{spitzer}, see also \cite{liggett}. Therefore, this approach justifies rigorously a method often used by physicists to establish the partial differential equations that describe the evolution of the thermodynamic characteristics of a fluid, and thus, the existence of weak solutions of such PDEs can be viewed as one of the goals of the hydrodynamic limit.

The random environment we considered is governed by the coefficients of the discrete formulation of the model (the process on the lattice). It is possible to obtain other formulations of random environments, for instance, in \cite{fjl} they proved a hydrodynamic limit for a gradient process with conductances in a random environment whose randomness consists of the random choice of the conductances. The hydrodynamic limit for a gradient process without conductances on the random environment we are considering was proved in \cite{gj}. We would like to mention that in \cite{f} a process evolving on a percolation cluster (a lattice with some bonds removed randomly) was considered and the resulting process turned out to be non-gradient. However, the homogenization tools facilitated the proof of the hydrodynamic limit, which made the proof much simpler than the usual proof of hydrodynamic limit for non-gradient processes (see for instance \cite[Chapter 7]{kl}).

We now describe the organization of the article. In Section \ref{sec2} we define the $W$-Sobolev spaces and obtain some results, namely, approximation by smooth functions, Poincaré's inequality, Rellich-Kondrachov theorem (compact embedding), and a characterization of the dual of the $W$-Sobolev spaces. In Section \ref{sec3} we define the $W$-generalized elliptic equations, and what we call by weak solutions. We then obtain some energy estimates and use them together with Lax-Milgram's theorem to conclude results regarding existence, uniqueness and boundedness of such weak solutions. In Section \ref{sec4} we define the $W$-generalized parabolic equations, their weak solutions, and prove uniquenesses of these weak solutions. Moreover, a notion of energy is also introduced in this Section. Section \ref{sec5} consists in obtaining discrete analogous results to the ones of the previous sections. This Section serves as preamble for the subsequent sections. In Section \ref{sec6} we define the random operators we are interested and obtain homogenization results for them. Finally, Section \ref{aplicacao-limite} concludes the article with an application that is interesting for both probability and theoretical physics, which is the hydrodynamic limit for a gradient process in random environments. This application uses results from all the previous sections and provides a proof for existence of weak solutions of $W$-generalized parabolic equations.

\section{$W$-Sobolev spaces}
\label{sec2}
This Section is devoted to the definition and derivation of properties of the $W$-Sobolev spaces. We begin by introducing some notation, stating some known results, and giving a precise definition of these spaces in subsection 2.2. Subsection 2.3 contains the proof of an approximation result. Poincaré's inequality, Rellich-Kondrachov theorem and a characterization of the dual space of these Sobolev spaces are also obtained. 

Denote by $\bb T^d = ({\bb R}/{\bb Z})^d = [0, 1)^d$ the $d$-dimensional torus, and by
$\bb T^d_N=(\bb Z/N\bb Z)^d = \{0,\ldots,N-1\}^d$ the $d$-dimensional discrete torus with $N^d$
points.

Fix a function $W: \bb R^d \to \bb R$ such that
\begin{equation}
\label{w}
W(x_1,\ldots,x_d) = \sum^d_{k=1}W_k(x_k),
\end{equation}
where each $W_k: \bb R \to \bb R$ is a \emph{strictly increasing} right continuous function with left
limits (c\`adl\`ag),
 periodic in the sense that for all $u\in \bb R$
 $$W_k(u+1) - W_k(u) = W_k(1) - W_k(0).$$  \\
 
Define the generalized derivative $\partial_{W_k}$ of a function $f:\bb T^d \to \bb R$ by
\begin{equation}
\label{f004}
\partial_{W_k} f (x_1,\!\ldots\!,x_k,\ldots, x_d) = \lim_{\epsilon\rightarrow 0}
\frac{f(x_1,\!\ldots\!,x_k +\epsilon,\ldots, x_d)
-f(x_1,\!\ldots\!,x_k,\!\ldots\!, x_d)}{W_k(x_k+\epsilon) -W_k(x_k)}\;,
\end{equation} 
when the above limit exists and is finite. If for a function $f:\bb T^d\to\bb R$ the generalized derivatives $\partial_{W_k}$ exist for all $k$, denote the generalized gradient of $f$ by
$$\nabla_W f = \left(\partial_{W_1}f,\ldots,\partial_{W_d}f\right).$$

Consider the operator $\mc L_{W_k}: \mc D_{W_k}\subset L^2(\bb T) \rightarrow \bb R$ given by
\begin{equation}
\label{f008}
\mc L_{W_k}f\;=\;  \partial_{x_k} \, \partial_{W_k} \, f,\;
\end{equation}
whose domain $\mc D_{W_k}$ is completely characterized in the following proposition:
\begin{proposition}
\label{dominiodw}
The domain $\mc D_{W_k}$ consists of all functions $f$ in $L^2(\bb T)$ such
that
\begin{equation*}
f(x) \;=\; a \;+\; b W_k(x) \;+\; \int_{(0,x]} W_k(dy) \int_0^y \mf f(z) \, dz
\end{equation*}
for some function $\mf f$ in $L^2(\bb T)$ that satisfies
\begin{equation*}
\int_0^1 \mf f(z) \, dz \;=\; 0\quad\hbox{~and~} \quad
\int_{(0,1]} W_k(dy) \Big\{ b + \int_0^y \mf f(z) \, dz \Big\} \;=\;0\; .
\end{equation*}
\end{proposition}
The proof of Proposition \ref{dominiodw} and further details can be found in \cite{TC}.  Furthermore, they also proved that these operators have a  countable complete orthonormal system of eigenvectors, which we denote by $\mc A_{W_k}$. Then, following \cite{v},
$$\mc A_W\;=\;\{f: \bb T^d \rightarrow \bb R;f(x_1,\ldots , x_d)=\prod^{d}_{k=1}{ f_k(x_k)}, f_k \in  \mc A_{W_k } \},$$
where $W$ is given by \eqref{w}. 

We may now build an operator analogous to $\mc L_{W_k}$ in $\bb T^d$. For a given set $\mc A$, we denote by $span(\mc A)$ the linear subspace generated by $\mc A$. Let $\bb D_W = span(\mc A_W)$, and define the operator $\bb L_W:\bb D_W \to L^2(\bb T^d)$ as follows: for $f =\prod^d_{k=1}{f_k}\in \mc A_W$,
\begin{equation}
\label{eq31}
\bb L_W(f)(x_1,\ldots x_d)=\sum^{d}_{k =1} \prod^{d}_{j=1, j\neq k}
{f_j(x_j)}\mc L_{W_k}f_k(x_k),
\end{equation}
and extend to $\bb D_W$ by linearity. It is easy to see that if $f\in \bb D_W$ 
\begin{equation}
\label{f002}
\bb L_W f=\sum^{d}_{k =1}\mc L_{W_k}f,
\end{equation}
where the application of $\mc L_{W_k}$ on a function $f:\bb T^d\to\bb R$ is the natural one, i.e., it considers $f$ only as a function of the $k$th coordinate, and keeps all the remaining coordinates fixed. 

Let, for each $k=1,\ldots,d$, $f_k\in \mc A_{W_k}$ be an eigenvector
of $\mc L_{W_k}$ associated to the eigenvalue $\lambda _k$. Then $f=\prod^d_{k=1}{f_k}$ belongs to $\bb D_W$ and is an eigenvector of
 $\bb L_W$ with eigenvalue $\sum^d_{k=1}{\lambda _k}$. Moreover, \cite{v} proved the following result: 
 
\begin{lemma}
\label{f17}
The following statements hold:

\renewcommand{\theenumi}{\alph{enumi}}
\renewcommand{\labelenumi}{{\rm (\theenumi)}}

\begin{enumerate}
\item The set $\bb D_W$ is dense in $L^2(\bb T^d)$;
\item The operator $\bb L_W : \bb D_W \to L^2(\bb T^d)$ is symmetric and
  non-positive:
\begin{eqnarray*}
\< -\bb L_W f , f\> \;\ge\; 0,
\end{eqnarray*}
where $\<\cdot,\cdot\>$ is the standard inner product in $L^2(\bb T^d)$.
\end{enumerate}
\end{lemma}

\subsection{The auxiliary space} 
Let $L^2_{x^k\x W_k}(\bb T^d)$ be the Hilbert space of
measurable functions $H: \bb T^d\to\bb R$ such that
\begin{equation*}
\int_{\bb T^d}d(x^k\x W_k) \, H (x)^2 \;<\; \infty,
\end{equation*}
where $d(x^k\x W_k)$  represents the product measure in $\bb T^d$
obtained from Lesbegue's measure in $\bb T^{d-1}$ and the measure induced by $W_k$ in $\bb T$:
$$d(x^k\x W_k)\;=\;dx_1\cdots dx_{k-1}\;dW_k\; dx_{k+1}\cdots dx_d.$$
Denote by $\< H,G \>_{x^k\x W_k}$ the inner product of $L^2_{x^k\x W_k}(\bb T^d)$:
\begin{equation*}
\< H,G \>_{x^k\x W_k} \;=\; \int_{\bb T^d} d(x^k\x W_k) \, H(x) 
\, G(x)\;,
\end{equation*}
 and by $\|\cdot\|_{x^k\x W_k}$ the norm induced by this inner product.
\begin{lemma}\label{mudavari}
Let $f,g\in \bb D_W$, then for $i=1,\ldots,d$,
$$\int_{\bb T^d} \big(\partial_{x_i}\partial_{W_i} f(x)\big)g(x)\ dx \;=\; -\int_{\bb T^d} (\partial_{W_i}f)(\partial_{W_i}g)d(x^i\x W_i).$$
In particular,
\begin{equation*}
\int_{\bb T^d} \bb L_W f(x)g(x)\ dx \;=\; -\sum_{i=1}^d \int_{\bb T^d} (\partial_{W_i}f)(\partial_{W_i}g)d(x^i\x W_i).
\end{equation*}
\end{lemma}
\begin{proof}
Let $f,g\in \bb D_W$. By Fubini's theorem
$$\int_{\bb T^d} \mc L_{W_i} f(x)g(x) dx = \int_{\bb T^{d-1}}\left[\int_{\bb T} \mc L_{W_i} f(x)g(x)dx_i\right] dx^i,$$
where $dx^i$ is the Lebesgue product measure in $\bb T^{d-1}$ on the coordinates $x_1,$ $\ldots,$ $x_{i-1},$ $x_{i+1},$ $\ldots,$ $x_d$.

An application of \cite[Lemma 3.1 (b)]{TC} and  again Fubini's theorem concludes the proof of this Lemma.
\end{proof}

Let $L^2_{x^j\x W_j, 0}(\bb T^d)$ be the closed subspace of $L^2_{x^j\x W_j}(\bb T^d)$ consisting of the functions that have zero mean with respect to the measure $d(x^j\x W_j)$:
$$\int_{\bb T^d} f d(x^j\x W_j) = 0.$$

Finally, using the characterization of the functions in $\mc D_{W_j}$ given in Proposition \ref{dominiodw}, and the definition of $\bb D_W$, we have that the set $\{\partial_{W_j} h; h\in\bb D_W\}$ is dense in $L^2_{x^j\x W_j, 0}(\bb T^d)$.

\subsection{The $W$-Sobolev space}
We define the Sobolev space of $W$-generalized derivatives as the space of functions $g\in L^2(\bb T^d)$ such that for each $i=1,\ldots,d$ there exist fuctions $G_i\in L^2_{x^i\x W_i,0}(\bb T^d)$ satisfying the following integral by parts identity.
\begin{equation}\label{eq22}
\int_{\bb T^d}\big(\partial_{x_i}\partial_{W_i}f\big)\;g\;dx\; =\; -\;\int_{\bb T^d}(\partial_{W_i}f)\;G_id(x^i\x W_i),
\end{equation}
for every function $f\in \bb D_W$. We denote this space by $\tilde{H}_{1,W}(\bb T^d)$. A standard measure-theoretic argument allows one to prove that for each function $g\in\tilde{H}_{1,W}(\bb T^d)$ and $i=1,\ldots,d$, we have a unique function $G_i$ that satisfies \eqref{eq22}. Note that $\bb D_W\subset \tilde{H}_{1,W}(\bb T^d)$. Moreover, if $g\in\bb D_W$ then $G_i = \partial_{W_i} g$. For this reason for each function $g\in \tilde{H}_{1,W}$ we denote $G_i$ simply by $\partial_{W_i} g$, and we call it the $i$th \emph{generalized weak derivative} of the function $g$ with respect to $W$.

\begin{lemma}\label{sobolevhilbert}
 The set $\tilde{H}_{1,W}(\bb T^d)$ is a Hilbert space with respect to the inner product
\begin{equation}\label{inner}
\<f,g\>_{1,W}\;=\; \<f,g\> + \sum_{i=1}^d\int_{\bb T^d}(\partial_{W_i}f)(\partial_{W_i}g)\;d(x^i\x W_i)
\end{equation}
\end{lemma}
\begin{proof}
Let $(g_n)_{n\in\bb N}$ be a Cauchy sequence in $\tilde{H}_{1,W}(\bb T^d)$, and denote by $\|\cdot\|_{1,W}$ the norm induced by the inner product \eqref{inner}. By the definition of the norm $\|\cdot\|_{1,W}$, we obtain that $(g_n)_{n\in\bb N}$ is a Cauchy sequence in $L^2(\bb T^d)$ and that $(\partial_{W_i}g_n)_{n\in\bb N}$ is a Cauchy sequence in $L^2_{x^i\x W_i,0}(\bb T^d)$ for each $i=1,\ldots,d$. Therefore, there exist functions $g\in L^2(\bb T^d)$ and $G_i\in L^2_{x^i\x W_i,0}(\bb T^d)$ such that $g=\lim_{n\to\infty} g_n$, and $G_i = \lim_{n\to\infty} \partial_{W_i}g_n$. It remains to be proved that $G_i$ is, in fact, the $i$th generalized weak derivative of $g$ with respect to $W$. But this follows from a simple calculation: for each $f\in\bb D_W$ we have
\begin{eqnarray*}
\int_{\bb T^d} \big(\partial_{x_i}\partial_{W_i}f\big) g dx &=& \lim_{n\to\infty} \int_{\bb T^d} \big(\partial_{x_i} \partial_{W_i}f\big) g_n dx\\
&=& -\lim_{n\to\infty} \int_{\bb T^d} (\partial_{W_i}f)( \partial_{W_i}g) d(x^i\x W_i)\\
&=& -\int_{T^d} (\partial_{W_i}f) G_i d(x^i\x W_i),
\end{eqnarray*}
where we used H\" older's inequality to pass the limit through the integral sign.
\end{proof}

\subsection{Approximation by smooth functions and the energetic space}
We will now obtain approximation of functions in the Sobolev space $\tilde{H}_{1,W}(\bb T^d)$ by
functions in $\bb D_W$. Note that the functions in $\bb D_W$ can be seen as smooth, in the sense that one may apply the operator $\bb L_W$ to these functions in the strong sense. 

Let us introduce $\<\cdot, \cdot\>_{1,W}$ the inner product on $\bb D_W$
defined by
\begin{equation}
\label{f51}
\<f,g\>_{1,W}\; =\; \<f,g\> \; +\; \<- \bb L_W f, g\>,
\end{equation}
and note that by Lemma \ref{mudavari},
\begin{equation*}
\<f,g\>_{1,W}\; =\; \<f,g\> \; +\; \sum_{i=1}^d \int_{\bb T^d} (\partial_{W_i}f)(\partial_{W_i}g)d(x^i\x W_i).
\end{equation*}

Let $H_{1,W} (\bb T)$ be the set of all functions $f$ in $L^2(\bb T^d)$
for which there exists a sequence $(f_n)_{n\in\bb N}$ in $\bb D_W$
such that $f_n$ converges to $f$ in $L^2(\bb T^d)$ and $f_n$ is a Cauchy sequence
for the inner product $\<\cdot, \cdot \>_{1,W}$. Such sequence
$(f_n)_{n\in\bb N}$ is called \emph{admissible} for $f$. 

For $f$, $g$ in $H_{1,W} (\bb T^d)$, define
\begin{equation}
\label{f20}
\<f,g\>_{1,W}\; =\; \lim_{n\to\infty} \<f_n,g_n\>_{1,W}\;,
\end{equation}
where $(f_n)_{n\in\bb N}$, $(g_n)_{n\in\bb N}$ are admissible sequences for $f$, and $g$,
respectively.  By \cite[Proposition 5.3.3]{z}, this limit exists and
does not depend on the admissible sequence chosen; the set $\bb D_W$ is dense in
$H_{1,W}$; and the embedding $H_{1,W} \subset L^2(\bb T^d)$ is continuous.  
 Moreover, $H_{1,W}(\bb T^d)$ endowed with the inner product $\<\cdot, \cdot \>_{1,W}$
just defined is a Hilbert space. Denote $\left\|\cdot \right\|_{1,W}$ the norm in $H_{1,W}$  induced by $\<\cdot, \cdot \>_{1,W}$.
The space $H_{1,W}(\bb T^d)$ is called energetic space. For more details on the theory of energetic spaces see \cite[Chapter 5]{z}.

Note that $H_{1,W}$ is the space of functions that can be approximated by functions in $\bb D_W$ with respect to the norm $\|\cdot\|_{1,W}$. The following Proposition shows that this space is, in fact, the Sobolev space $\tilde{H}_{1,W}(\bb T^d)$.

\begin{proposition}[Approximation by smooth functions]\label{aproxsuave}
We have the equality of the sets
$$\tilde{H}_{1,W}(\bb T^d) = H_{1,W}(\bb T^d).$$
In particular, we can approximate any function $f$ in the Sobolev space $\tilde{H}_{1,W}(\bb T^d)$ by functions in $\bb D_W$.
\end{proposition}
\begin{proof}
Fix $g\in H_{1,W}(\bb T^d)$. By definition, there exists a sequence $g_n$ in $\bb D_W$ such that $g_n$ converges to $g$ in $L^2(\bb T^d)$ and $g_n$ is Cauchy for the inner product $\<\cdot, \cdot \>_{1,W}$.
So, for each $i=1,\ldots, d$ there exists functions $G_i\in L^2_{x^i\x W_i,0}(\bb T^d)$ such that $\partial_{W_i}g_n$ converges to $G_i$ in $L^2_{x^i\x W_i,0}(\bb T^d)$. Applying the Hölder's inequality, we deduce that for every $f\in \bb D_W$
\begin{equation*}
 \int_{\bb T^d}\big(\partial_{x_i} \partial_{W_i}f\big)g\;dx\;=\;
 \lim_{n\to\infty}\int_{\bb T^d}\big(\partial_{x_i} \partial_{W_i}f\big)g_n\;dx.
\end{equation*}
By Lemma \ref{mudavari}, we obtain
\begin{eqnarray*}
\lim_{n\to\infty}\int_{\bb T^d}\big(\partial_{x_i} \partial_{W_i} f\big)g_n dx &=& \lim_{n\to\infty}\int_{\bb T^d}(\partial_{W_i} f)(\partial_{W_i}g_n)\;d(x^i\x W_i)\\
& =&-
\int_{\bb T^d}(\partial_{W_i}f)G_i\;d(x^i\x W_i).
\end{eqnarray*}
Then, $g\in \tilde{H}_{1,W}(\bb T^d)$ and therefore $H_{1,W}(\bb T^d)\subset \tilde{H}_{1,W}(\bb T^d)$.

We will now prove that $H_{1,W}(\bb T^d)$ is dense in $\tilde{H}_{1,W}(\bb T^d)$, and since both of them are complete, they are equal. Note that
since $\bb D_W$ is dense in $L^2(\bb T^d)$ and $\bb D_W \subset H_{1,W}(\bb T^d)$, we have that $H_{1,W}(\bb T^d)$ is also dense in $L^2(\bb T^d)$.

Therefore, given a function $g\in \tilde{H}_{1,W}(\bb T^d)$, we can approximate $g$ by a sequence of functions $(f_n)_{n\in \bb N}$ in ${H}_{1,W}(\bb T^d)$  with respect to the $L^2(\bb T^d)$ norm. Let $F_{i,n}$ be the $i$th generalized weak derivative of $f_n$ with respect to $W$. We have, therefore, for each $h\in \bb D_W$
$$\lim_{n\to\infty}\int_{\bb T^d} (\partial_{W_i}h)(F_{i,n} - G_i) d(x^i\x W_i) = -\lim_{n\to\infty}\int_{\bb T^d} \big(\partial_{x_i}\partial_{W_i}h\big) (f_n - g)dx = 0.$$

Denote by $\mc F_{i,n}:L^2_{x^i\x W_i,0}(\bb T^d)\to \bb R$ the sequence of bounded linear functionals induced by $F_{i,n} - G_i$:
$$\mc F_{i,n} (h) := \int_{\bb T^d} h [F_{i,n} - G_i] d(x^i\x W_i),$$
for $h\in L^2_{x^i\x W_i,0}(\bb T^d)$. We then note that, since the set $\{\partial_{W_i}h; h\in \bb D_W\}$ is dense in $L^2_{x^i\x W_i,0}(\bb T^d)$, $\mc F_{i,n}$ converges to $0$ pointwisely. By Banach-Steinhaus' Theorem, $\mc F_{i,n}$ converges strongly to $0$, and, thus, $F_{i,n}$ converges to $G_i$ in $L^2_{x^i\x W_i,0}(\bb T^d)$, for each $i=1,\ldots,d$. Therefore, $f_n$ converges to $g$ in $L^2(\bb T^d)$ and $\partial_{W_i}f_n$ converges to $G_i$ in $L^2_{x^i\x W_i,0}(\bb T^d)$ for each $i$, i.e., $f_n$ converges to $g$ with the norm $\|\cdot\|_{1,W}$, and the density of $H_{1,W}(\bb T^d)$ in $\tilde{H}_{1,W}(\bb T^d)$ follows.
\end{proof}

The next Corollary shows an analogous of the classic result for Sobolev spaces with dimension $d=1$, which states that every function in the one-dimensional Sobolev space is absolutely continuous. 
\begin{corollary}\label{abscont}
A function $f$ in $L^2(\bb T)$ belongs to the Sobolev space $\tilde{H}_{1,W} (\bb T)$ if and only
if there exists $F$ in $L^2_W(\bb T)$ and a finite constant $c$ such that
\begin{equation*}
\int_{(0,1]} F(y) \, dW(y) \;=\;0 \quad\text{and}\quad
f(x) \;=\; c\;+\; \int_{(0,x]} F(y) \, dW(y)
\end{equation*}
Lebesgue almost surely. 
\end{corollary}
\begin{proof}
In \cite{TC} the energetic extension $H_{1,W}(\bb T)$ has the characterization given in Corollary \ref{abscont}. By Proposition \ref{aproxsuave} we have that these spaces coincide, and hence the proof follows.
\end{proof}

From Proposition \ref{aproxsuave}, we may use the notation $H_{1,W}(\bb T^d)$ for the Sobolev space $\tilde{H}_{1,W}(\bb T^d)$.
Another interesting feature we have on this space, which is very useful in the study of elliptic equations, is the Poincar\'e inequality:
\begin{corollary}[Poincaré Inequality]\label{poinc}
For all $f\in {H}_{1,W}(\bb T^d)$ there exists a finite constant $C$ such that
\begin{eqnarray*}
\left\|f-\int_{\bb T^d}f\;dx\right\|_{L^2(\bb T^d)}^2 &\le&C\sum_{i=1}^n \int_{\bb T^d} \left(\partial_{W_i}f\right)^2 d(x^i\x W_i)\\
&:=& C \|\nabla_W f\|_{L_W^2(\bb T^d)}^2.
\end{eqnarray*}
\end{corollary}
\begin{proof}
We begin by introducing some notations. For $x,y\in\bb T^d$, $i=0,\ldots,d$ and $t\in\bb T$, denote
$$z(x,y,i) = (x_1,\ldots,x_{d-i},y_{d-i+1},\ldots,y_d)\in\bb T^d$$
and
$$z(x,y,t,i) = (x_1,\ldots,x_{d-i},t,y_{d-i+2},\ldots,y_d)\in\bb T^d.$$

With this notation, we may write $f(x)-f(y)$ as the telescopic sum
$$f(x) - f(y) = \sum_{i=1}^d f(z(x,y,i-1)) - f(z(x,y,i)).$$
We are now in conditions to prove this Lemma. Let $f\in\bb D_W$, then
\begin{align*}
\Big\|f &- \int_{\bb T^d} fdx\Big\|_{L^2(\bb T^d)}^2 = \int_{\bb T^d} \Big[ \int_{\bb T^d} f(x) - f(y) dy\Big]^2 dx\\
&= \int_{\bb T^d} \Big[ \int_{\bb T^d} \sum_{i=1}^d \int_{y_i}^{x_i} \partial_{W_i}f(z(x,y,t,i)) dW_i(t) dy\Big]^2 dx\\
&\leq \int_{\bb T^d}\Big[ \int_{\bb T^d} \sum_{i=1}^d \int_{\bb T} \Big| \partial_{W_i}f(z(x,y,t,i)) \Big| dW_i(t) dy\Big]^2dx\\
&\leq \int_{\bb T^d} \Big[ \sum_{i=1}^d \int_{\bb T^{d-i+1}} \Big| \partial_{W_i}f(z(x,y,t,i))\Big| dW_{d-i}(t)\x y_{d-i+1}\x\cdots\x y_d\Big]^2dx\\
&\leq C \int_{\bb T^d} \sum_{i=1}^d \int_{\bb T^{d-i+1}} \Big|\partial_{W_i}f(z(x,y,t,i))\Big|^2 dW_{d-i}(t)\x dy_{d-i+1}\x\cdots\x dy_d dx\\
&= C\sum_{i=1}^d \int_{\bb T^d} \Big(\partial_{W_i}f\Big)^2 d(x^i\x W_i),
\end{align*}
where in the next-to-last inequality, we used Jensen's inequality and the elementary inequality $(\sum_i x_i)^2\leq C \sum_i x_i^2$ for some positive constant $C$.
To conclude the proof, one uses Proposition \ref{aproxsuave} to approximate functions in $H_{1,W}(\bb T^d)$ by functions in ${\bb D}_W$.
\end{proof}

\subsection{A Rellich-Kondrachov theorem}
In this subsection we prove an analogous of the Rellich-Kondrachov theorem for the $W$-Sobolev spaces. We begin by stating this result in dimension $1$, whose proof can be found in \cite[Lemma 3.3]{TC}.

\begin{lemma}\label{rk-1}
Fix some $k\in\{1,\ldots,d\}$. The embedding $H_{1,W_k}(\bb T) \subset L^2(\bb T)$ is compact.
\end{lemma}
Recall that they proved this result for the energetic extension, but in view of Proposition \ref{aproxsuave}, this result holds for our Sobolev space $H_{1,W_k}(\bb T)$.

\begin{proposition}[Rellich-Kondrachov]
\label{f3}
The embedding $H_{1,W}(\bb T^d) \subset \bb L^2(\bb T^d)$ is compact.
\end{proposition}
\begin{proof}
We will outline the strategy of the proof. Using the definition of the set $\bb D_W$ and the fact that it is dense in $H_{1,W}(\bb T^d)$, it is enough to show this fact for sequences in $\bb D_W$. From this point, the main tool is Lemma \ref{rk-1} and Cantor's diagonal method to obtain converging subsequences.

We begin by noting that by Proposition \ref{aproxsuave}, it is enough to prove that the embed $\bb D_W\subset L^2(\bb T^d)$ is compact.

Let $C>0$ and consider a sequence $(v_n)_{n\in\bb N}$ in $\bb D_W$, with $\|v_n\|_{1,W}\leq C$ for all $n\in\bb N$. We have, by definition of $\bb D_W$ (see the definition at the beginning of Section \ref{sec2}), that each $v_n$ can be expressed as a finite linear combination of elements in $\mc A_W$. Furthermore, each element in $\mc A_W$ is a product of elements in $\mc A_{W_k}$ for $k=1,\ldots,d.$ Therefore, we can write $v_n$ as 
$$v_n = \sum^{N(n)}_{j=1}{\alpha^n_j\prod^d_{k=1}{g^n_{k,j}}} = \sum^{N(n)}_{j=1}{\alpha^n_jg^n_{j}},$$
where $g^n_{k,j}\in\mc A_{W_k},\alpha^n_j\in\bb R, g_j^n = \prod_{k=1}^d g_{j,k}^n,$ and $N(n)$ is chosen such that $N(n)\geq n$ (we can complete with zeros if necessary). Recall that these functions $g^n_{k,j}$ have $\|g^n_{k,j}\|_{L^2(\bb T)} = 1$, and hence, $\|g_j^n\|_{L^2(\bb T^d)}=1$. Moreover, the set $\{g_1^n,\ldots,g_{N(n)}^n\}$ is orthogonal in $L^2(\bb T^d)$.

From orthogonality, we obtain that 
$$\sum_{j=1}^{N(n)} (\alpha_j^n)^2 \leq C^2,\qquad\hbox{uniformly in~} n\in \bb N.$$
 
Note that the uniform boundedness of $v_n$ in $H_{1,W}(\bb T^d)$ implies the uniform boundedness of $\|g_{k,j}^n\|_{1,W_k}$, for all $k=1,\ldots,d$, $j=1,\ldots,N(n)$ and $n\in\bb N$. Our goal now is to apply Lemma \ref{rk-1} to our current setup.

Consider the sequence of functions $\alpha_1^n g_{1,1}^n$ in $H_{1,W_1}(\bb T)$. By Lemma \ref{rk-1}, this sequence has a converging subsequence, and we call the limit point $\alpha_1 g_{1,1}$. Repeat this step $d-1$ times for the sequences $g_{k,1}^n$ in $H_{1,W_k}(\bb T)$, for $k=2,\ldots,d$, considering in each step a subsequence of the previous step, to obtain converging subsequences, and call their limit points $g_{k,1}$. At the end of this procedure, we obtain a converging subsequence of $\prod_{k=1}^d\alpha_1^n g_{1,k}^n$, with limit point $\prod_{k=1}^d \alpha_1 g_{1,k}\in L^2(\bb T^d)$, which we will denote by $\alpha_1 g_1$. 

In the $j$th step, in which we want to obtain the limit point $\alpha_j g_j$, we repeat the previous idea, with the sequences $\alpha_j^n g_{j,1}^n$ and $g_{j,k}^n$, with $n\leq j$ and $k=2,\ldots,d$. We note that it is always necessary to consider a subsequence of all the previous steps.

This procedure provides limiting functions $\alpha_j g_j$, for all $j\in\bb N$. From now on, we use the notation $v_n$ to mean the diagonal sequence obtained to ensure the convergence of the functions $\alpha_j^n g_j^n$ to $\alpha_j g_j$. We claim that the function
$$v = \sum_{j=1}^\infty \alpha_j g_j$$
is well-defined and belongs to $L^2(\bb T^d)$. To prove this claim, note that the set $\{g_k\}_{k\in\bb N}$ is orthonormal by the continuity of the inner product. Suppose that there exists $N\in\bb N$ such that
$$\sum_{j=1}^N (\alpha_j)^2 > C^2.$$
We have that the sequence of functions 
$$v_n^N := \sum_{j=1}^N \alpha_j^n g^n_{j}$$
converges to 
$$v^N := \sum_{j=1}^N \alpha_j g_{j}.$$
Since $\|v_n^N\| \leq C$ uniformly in $n\in\bb N$, this yields a contradiction. Therefore $v\in L^2(\bb T^d)$ with the bound $\|v\|\leq C$. 

It remains to be proved that $v_n$ has a subsequence that converges to $v$. Choose $N$ so large that $\|v-v^N\|<{\epsilon}/{3}$, $\|v_n^N -v^N\|<\epsilon/3$ and $\|v_n^N - v_n\|<\epsilon/3$, and use the triangle inequality to conclude the proof.

\end {proof}

\subsection{The space $H^{-1}_W(\bb T^d)$}\label{dual2}
Let $H^{-1}_W(\bb T^d)$ be the dual space to $H_{1,W}(\bb T^d)$, that is, $H^{-1}_W(\bb T^d)$ is the set of bounded linear functionals on $H_{1,W}(\bb T^d)$. Our objective in this subsection is to characterize the elements of this space. This proof is based on the characterization of the dual of the standard Sobolev space in $\bb R^d$ (see \cite{E}).

We will write $(\cdot,\cdot)$ to denote the pairing between $H^{-1}_W(\bb T^d)$ and $H_{1,W}(\bb T^d)$.

\begin{lemma}
$f\in H^{-1}_{W}(\bb T^d)$ if and only if there exist functions 
$f_0\in L^2(\bb T^d),$ and $f_k\in L^2_{x^k\x W_k,0}(\bb T^d)$, such that
\begin{equation}\label{dual}
f = f_0 - \sum_{i=1}^d\partial_{x_i}f_i,
\end{equation}
in the sense that for $v\in H_{1,W}(\bb T^d)$
$$(f,v) = \int_{\bb T^d} f_0 v dx + \sum_{i=1}^d \int_{\bb T^d} f_i (\partial_{W_i}v) d(x^i\x W_i).$$
Furthermore,
$$\|f\|_{H^{-1}_W} = \inf\left\{\left(\int_{\bb T^d} \sum_{i=0}^d |f_i|^2 dx\right)^{1/2} ;\quad f\hbox{~satisfies \eqref{dual}}\right\}.$$
\end{lemma}
\begin{proof}
Let $f\in H^{-1}_W(\bb T^d)$. Applying the Riesz Representation Theorem, we deduce the existence of a unique function $u\in H_{1,W}(\bb T^d)$ satisfying $(f,v) = \<u,v\>_{1,W}$, for all $v\in H_{1,W}(\bb T^d)$, that is
\begin{equation}\label{carac}
\int_{\bb T^d} u v dx + \sum_{j=1}^d \int_{\bb T^d} (\partial_{W_j}u)( \partial_{W_j}v) d(x^j\x W_j)= (f,v),\quad\hbox{~~for all~~}v\in H_{1,W}(\bb T^d).
\end{equation}

This establishes the first claim of the Lemma for $f_0 = u$ and $f_i = \partial_{W_i} u$, for $i=1,\ldots,d$.

Assume now that $f\in H_{W}^{-1}(\bb T^d)$,
\begin{equation}\label{carac2}
(f,v) = \int_{\bb T^d} g_0 v dx + \sum_{i=1}^d \int_{\bb T^d} g_i (\partial_{W_i}v) d(x^i\x W_i),
\end{equation}
for $g_0, g_1,\ldots , g_d \in L^2_{x^j\x W_j,0}(\bb T^d)$. Setting $v=u$ in \eqref{carac}, using \eqref{carac2}, and applying the Cauchy-Schwartz inequality twice, we deduce
\begin{equation}\label{carac3}
\|u\|_{1,W}^2\le \int_{\bb T^d}g_0^2dx + \sum_{i=1}^d\int_{\bb T^d}\partial_{W_i}g_i^2d(x^i\x W_i).
\end{equation}
From \eqref{carac} it follows that
$$|(f,v)|\le \|u\|_{1,W}$$
if $\|v\|_{1,W}\le 1$. Consequently
$$\|f\|_{H^{-1}_W}\le \|u\|_{1,W}.$$
Setting $v = u/\|u\|_{1,W}$ in \eqref{carac}, we deduce that, in fact,
$$\|f\|_{H^{-1}_W}= \|u\|_{1,W}.$$
The result now follows from the above expression and equation \eqref{carac3}.
\end{proof}

\section{$W$-Generalized elliptic equations}
\label{sec3}
This subsection investigates the solvability of uniformly elliptic generalized partial differential equations defined below.
Energy methods within Sobolev spaces are, essentially, the techniques exploited.

Let $A=(a_{ii}(x))_{d\times d}$, $x\in \bb T^d$, be a diagonal matrix function such that there exists a constant $\theta>0$ satisfying
\begin{equation}\label{cota A}
\theta^{-1}\le a_{ii}(x) \le \theta,
\end{equation}
for every $x\in \bb T^d$ and $i=1,\ldots,d$. To keep notation simple, we write $a_i(x)$ to mean $a_{ii}(x)$.

Our interest lies on the study of the problem
\begin{equation}\label{prob T}
T_{\lambda}u = f,
\end{equation}
where $u:\bb T^d \to \bb R$ is the unknown function and $f:\bb T^d \to \bb R$ is given. Here $T_{\lambda}$ denotes the generalized elliptic operator
\begin{equation}\label{def T}
  T_{\lambda}u\; :=\;\lambda u - \nabla A\nabla_Wu\;:=\;\lambda u - \sum_{i=1}^d\partial_{x_i}\Big(a_i(x)\partial_{W_i} u\Big).
\end{equation}

The bilinear form $B[\cdot,\cdot]$ associated with the elliptic operator $T_{\lambda}$ is given by
\begin{equation}\label{def B}
B[u,v] = \lambda\<u,v\> + \sum_{i=1}^d\int a_i(x)(\partial_{W_i}u)(\partial_{W_i}v)\;d(W_i\otimes x_i),
\end{equation}
where $u,v\in H_{1,W}(\bb T^d)$.

Let $f\in H^{-1}_W(\bb T^d)$. A function $u\in H_{1,W}(\bb T^d)$ is said to be a weak solution of the equation
$T_\lambda u = f\;$ if 
\begin{equation*}
B[u,v]\;=\;(f,v) \;\;\text{for all}\;\;v\in H_{1,W}(\bb T^d).
\end{equation*}

Recall a classic result from linear functional analysis, which provides in certain circumstances the existence and
uniqueness of weak solutions of our problem, and whose proof can be found, for instance, in \cite{E}. Let $\mc H$ be a Hilbert space endowed with inner product  $<\!\cdot,\cdot\!>$ and norm $\||\cdot\||$. Also,  $(\cdot,\cdot)$ denotes the pairing of $\mc H$ with its dual space.
\begin{theorem}[Lax-Milgram Theorem] 
Assume that $\textit{B} : \mc H\times \mc H\to \bb R$ is a bilinear mapping on Hilbert space $\mc H$, for which there exist constants $\alpha>0$ and $\beta >0$ such that for all $u,v\in\mc H$,
\begin{equation*}
|\textit{B}[u,v]|\le \alpha\||u\||\cdot\||v\||\;\;\text{and}\;\;\textit{B}[u,u]\ge \beta\||u\||^2.
\end{equation*}

Let $f:\mc H\to\bb R$ be a bounded linear functional on $\mc H$. Then there exists a unique element $u\in\mc H$ such that
$$\textit{B}[u,v] = (f,v),$$
for all $v\in\mc H$.
\end{theorem}

Return now to the specific bilinear form $B[\cdot,\cdot]$ defined in \eqref{def B}. Our goal now is to verify the hypothesis of Lax-Milgram Theorem for our setup. We consider the cases $\lambda=0$ and $\lambda>0$ separately. We begin by analyzing the case in which $\lambda =0$.

Let $H^\bot_{1,W}(\bb T^d)$ be the set of functions in $H_{1,W}(\bb T^d)$ which are orthogonal to the constant functions:
\begin{equation*}
H^\bot_{1,W}(\bb T^d)\;=\;\{f\in H_{1,W}(\bb T^d); \int_{\bb T^d}f\;dx =0\}.
\end{equation*}

The space $H^\bot_{1,W}(\bb T^d)$ is the natural environment to treat elliptic operators with Neumann condition.

\begin{proposition}[Energy estimates for $\lambda =0$]\label{lm-0}
Let $\textit{B}$ be the bilinear form on $H_{1,W}(\bb T^d)$ defined in \eqref{def B} with $\lambda =0$. There exist constants 
$\alpha>0$ and $\beta >0$ such that for all $u,v\in H_{1,W}(\bb T^d)$,
\begin{equation*}
|\textit{B}[u,v]|\le \alpha\|u\|_{1,W}\;\|v\|_{1,W}
\end{equation*}
and for all $u\in H^\bot_{1,W}$ 
$$ \textit{B}[u,u]\ge \beta\|u\|^2_{1,W}.$$
\end{proposition}
\begin{proof}
By \eqref{cota A}, the computation of the upper bound $\alpha$ easily follows. For the lower bound $\beta$, we have for $u\in H^{\bot}_{1,W}(\bb T^d)$,
\begin{equation*}
\|u\|^2_{1,W} = \int_{\bb T^d}u^2\;dx + \sum_{i=1}^d\int_{\bb T^d}\Big(\partial_{W_i} u\Big)^2d(x^i\otimes W_i).
\end{equation*}
Using Poincaré's inequality and \eqref{cota A}, we obtain a constant $C>0$ such that the previous expression is bounded above by
$$C\int_{\bb T^d}\Big(\partial_{W_i} u\Big)^2d(x^i\otimes W_i)\le C\textit{B}[u,u].$$
The lemma follows from the previous estimates.
\end{proof}

\begin{corollary}
Let $f\in L^2(\bb T^d)$. There exists a weak solution $u\in H_{1,W}(\bb T^d)$ for the equation 
\begin{equation}\label{pc}
\nabla A\nabla_Wu\;=\;f
\end{equation}
if and only if 
\begin{equation*}
\int_{\bb T^d}f dx\;=\;0.
\end{equation*}
In this case, we have uniquenesses of the weak solutions if we disregard addition by constant functions. Also, let $u$ be the unique weak solution of \eqref{pc} in $H^\bot_{1,W}(\bb T^d)$. Then
$$\|u\|_{1,W}\le C\|f\|_{L^2(\bb T^d)},$$
for some constant $C$ independent of $f$.
\end{corollary}
\begin{proof}
Suppose that there exists a weak solution $u\in H_{1,W}(\bb T^d)$ of \eqref{pc}. Since the function $v\equiv 1 \in H_{1,W}(\bb T^d)$, we have by definition of weak solution that
\begin{equation*}
 \int_{\bb T^d} f dx\;=\;\textit{B}[u,v]\; =\;0 .
\end{equation*} 

Now, let $f\in L^2(\bb T^d)$ with $\int_{\bb T^d} f dx=0$. Consider the bilinear form $\textit{B}$, defined in \eqref{def B} with $\lambda =0$, on the Hilbert space $H^\bot_{1,W}(\bb T^d)$. By Proposition \ref{lm-0}, $\textit{B}$ satisfies the hypothesis of the Lax-Milgram's Theorem. Further, $f$ defines the bounded linear functional in $H^\bot_{1,W}(\bb T^d)$ given by $(f,g)=\<f,g\>$ for every $g\in H^\bot_{1,W}(\bb T^d)$. Then, an application of Lax-Milgram's Theorem yields that there exists a unique $u\in H^\bot_{1,W}(\bb T^d)$ such that  
$$
\textit{B}[u,v] = \<f,v\>\; \text{for all}\;v\in H^\bot_{1,W}(\bb T^d).
$$
Moreover, by Proposition \ref{lm-0}, there is a $\beta>0$ such that
$$
\beta\|u\|_{1,W}^2\le B[u,u] =\<f,u\>\le\| f\|_{L^2(\bb T^d)} \| u\|_{L^2(\bb T^d)}\le 
\| f\|_{L^2(\bb T^d)} \| u\|_{1,W}.
$$

The existence of weak solutions and the bound $C$ in the statement of the Corollary follows from the previous expression.
\end{proof}

We now analyze the case in which $\lambda>0$.
\begin{proposition}[Energy estimates for $\lambda >0$]\label{lm-lambda}
Let $f\in L^2(\bb T^d)$. There exists a unique weak solution $u\in H_{1,W}(\bb T^d)$ for the equation 
\begin{equation}\label{prob b}
\lambda u - \nabla A\nabla_Wu\;=\;f,\qquad\lambda >0.
\end{equation}
This solution enjoys the following bounds
$$\|u\|_{1,W}\le C\|f\|_{L^2(\bb T^d)}$$
for some constant $C>0$ independent of $f$, and
$$\|u\| \leq \lambda^{-1} \|f\|_{L^2(\bb T^d)}.$$
\end{proposition}
\begin{proof}
Let $\beta = min\{\lambda,\theta^{-1}\}>0$ and 
$\alpha = max\{\lambda,\theta\}<\infty$, where $\theta$ is given in \eqref{cota A}. An elementary computation shows that
\begin{align*}
\textit{B}&[u,v]|\le \alpha\|u\|_{1,W}\;\|v\|_{1,W}\;\;\;\text{and}\;\;\; \textit{B}[u,u]\ge \beta\|u\|^2_{1,W}.
\end{align*}

By Lax-Milgram's Theorem, there exists a unique solution $u\in H_{1,W}(\bb T^d)$ of \eqref{prob b}.
Note that
\begin{equation*}
\beta\|u\|_{1,W}^2\le B[u,u] =\<f,u\>\le\| f\|_{L^2(\bb T^d)} \| u\|_{L^2(\bb T^d)}\le 
\| f\|_{L^2(\bb T^d)} \| u\|_{1,W},
\end{equation*}
and therefore $\|u\|_{1,W}\le C\|f\|_{L^2(\bb T^d)}$ for some constant $C>0$ independent of $f$. The computation to obtain the other bound is analogous.
\end{proof}
\begin{remark}\label{a-adj}
Let $\bb L_W^A:\bb D_W\to\bb L^2(\bb T^d)$ be given by $\bb L_W^A =\nabla A\nabla_W$. This operator has the properties stated in Theorem 2.1 in \cite{v}. We now outline the main steps to prove it. Following \cite{v}, we may prove an analogous of Lemma \ref{f17} for the operator $\bb L_W^A$. Using the bounds on the diagonal matrix $A$ and Proposition \ref{f3} (Rellich-Kondrachov), we conclude that the energetic extension of the space induced by this operator has compact embedding in $L^2(\bb T^d)$. The previous results together with \cite[Theorems 5.5.a and 5.5.c]{z} implies that $\bb L_W^A$ has a self-adjoint extension $\mc L_W^A$, which is dissipative and non-positive, and its eigenvectors form a complete orthonormal set in $L^2(\bb T^d)$. Furthermore, the set of eigenvalues of this extension is countable and its elements can be ordered resulting in a non-increasing sequence that tends to $-\infty$. 
\end{remark}
\begin{remark}
Let $\mc L_W^A$ be the self-adjoint extension given in Remark \ref{a-adj}, and $\mc D_W^A$ its domain. For $\lambda>0$ the operator $\lambda \bb I - \mc L_W^A:\mc D_W\to L^2(\bb T^d)$ is bijective. Therefore, the equation 
$$\lambda u - \nabla A\nabla_W u = f,$$
has strong solution in $\bb D_W$ if and only if $f\in (\lambda \bb I - \mc L_W^A)(\bb D_W)$, where $\bb I$ is the identity operator and $(\lambda \bb I - \mc L_W^A)(\bb D_W)$ stands for the range of $\bb D_W$ under the operator $\lambda \bb I - \mc L_W^A$. Moreover, this strong solution coincides with the weak solution obtained in Proposition \ref{lm-lambda}.
\end{remark}
\section{$W$-Generalized parabolic equations}
\label{sec4}
In this Section, we study a class of $W$-generalized PDEs that involves time: the parabolic equations. The parabolic equations are often used to describe in physical applications the time-evolution of the density of some quantity, say a chemical concentration within a region. The motivation of this generalization is to enlarge the possibility of such applications, for instance, these equations may be used to model a diffusion of particles within a region with membranes (see \cite{TC,v}).

We begin by introducing the class of $W$-generalized parabolic equations we are interested. Then, we define what is meant by weak solution of such equations, using the $W$-Sobolev spaces, and prove uniquenesses of these weak solutions. In Section \ref{aplicacao-limite}, we obtain existence of weak solutions of these equations.

Fix $T>0$ and let $(B,\|\cdot\|_B)$ be a Banach space. We denote by $L^2([0,T],B)$ the Banach space of measurable functions $U:[0,T]\to B$ for which
$$\|U\|_{L^2([0,T],B)}^2 := \int_0^T \|U_t\|_B^2 dt <\infty.$$

Let $A = A(t,x)$ be a diagonal matrix satisfying the ellipticity condition \eqref{cota A} for all $t\in [0,T]$, $\Phi:[l,r]\to\bb R$ be a continuously differentiable function such that $$B^{-1}<\Phi'(x)<B,$$ for all $x$, where $B>0$, $l,r\in\bb R$ are constants. 
We will consider the equation
\begin{equation}\label{parabolic}
\left\{\begin{array}{cc}
\partial_t u = \nabla A\nabla_W \Phi(u)& \hbox{~in~}(0,T]\times\bb T^d,\\
u= \gamma&\hbox{~in~}\{0\}\times\bb T^d.
\end{array}
\right.
\end{equation}
where $u:\bb [0,T]\times T^d \to \bb R$ is the unknown function and $\gamma:\bb T^d \to \bb R$ is given.\\

We say that a function $\rho = \rho(t,x)$ is a weak solution of the problem \eqref{parabolic} if: 
\begin{itemize}
\item For every $H\in\bb D_W$ the following integral identity holds
\begin{eqnarray*}
\int_{\bb T^d} \rho(t,x) H(x)dx - \int_{\bb T^d} \gamma(x) H(x)dx=
\int_0^t \, \int_{\bb T^d} \Phi (\rho(s,x))  \nabla A\nabla_W H(x)dx\,ds \\
\end{eqnarray*}
\item $\Phi(\rho(\cdot,\cdot))$ and $\rho(\cdot,\cdot)$ belong to $L^2([0,T],H_{1,W}(\bb T^d))$:
$$\int_0^T \|\Phi(\rho(s,x))\|_{L^2(\bb T^d)}^2 + \|\nabla_{W}\Phi(\rho(s,x))\|_{L^2_W(\bb T^d)}^2 ds <\infty,$$
and
$$\int_0^T \|\rho(s,x)\|_{L^2(\bb T^d)}^2 + \|\nabla_{W}\rho(s,x)\|_{L^2_W(\bb T^d)}^2 ds <\infty.$$
\end{itemize}

Consider the energy in $j$th direction of a function $u(s,x)$ as
\begin{align*}
\mc Q_j(u) = \sup_{H\in\bb D_W} \Big\{ 2\int_0^T\, \int_{\bb T^d}
 (\partial_{x_j}&\partial_{W_j}H) (s, x) \, u(s,x)dx\, ds\\
-  &\int_0^T ds\, \int_{\bb T^d} [\partial_{W_j}H (s, x)]^2
 d(x^j\x W_j)\Big\},
\end{align*}
and the total energy of a function $u(s,x)$ as
$$\mc Q(u) = \sum_{j=1}^d \mc Q_j (u).$$

The notion of energy is important in probability theory and is often used in large deviations of Markov processes. We also use this notion to prove the hydrodynamic limit in Section \ref{aplicacao-limite}.
The following lemma shows the connection between the functions of finite energy and functions in the Sobolev space.
\begin{lemma}\label{energia}
A function $u\in L^2([0,T],L^2(\bb T^d))$ has finite energy if and only if $u$ belongs to $L^2([0,T],H_{1,W}(\bb T^d))$. In the case the energy is finite, we have
$$\mc Q(u) = \int_0^T \|\nabla_W u\|_{L^2_W(\bb T^d)}^2 dt.$$
\end{lemma}
\begin{proof}
Consider functions $U\in L^2([0,T],L^2_{x^j\x W_j,0}(\bb T^d))$ as trajectories in $L^2_{x^j\x W_j,0}(\bb T^d)$, that is, consider a trajectory $\bs U:[0,T]\to L^2_{x^j\x W_j,0}(\bb T^d)$ and define $U(s,x)$ as $U(s,x):=[\bs U(s)](x)$. 

Let $u\in L^2([0,T],L^2(\bb T^d))$ and recall that the set $\{\partial_{W_j}H; H\in\bb D_W\}$ is dense in $L^2_{x^j\x W_j,0}(\bb T^d)$. Then the set $\{\partial_{W_j}H(s,x); H\in L^2([0,T],\bb D_W)\}$ is dense in $L^2([0,T],L^2_{x^j\x W_j,0}(\bb T^d))$. Suppose that $u$ has finite energy, and let $H\in L^2([0,T],\bb D_W)$, then 
$$\mc F_j(\partial_{W_j}H) = \int_0^T \, \int_{\bb T^d}
 (\partial_{x_j}\partial_{W_j}H) (s,x) \, u(s,x)dx\, ds$$
is a bounded linear functional in $L^2([0,T],L^2_{x^j\x W_j,0}(\bb T^d))$. Consequently, by Riesz's representation theorem, there exists a function $G_j\in L^2([0,T],L^2_{x^j\x W_j,0}(\bb T^d))$ such that
$$\mc F_j(\partial_{W_j}H) =  \int_0^T\int_{\bb T^d}
 (\partial_{W_j}H) (x) \, G_j(s,x) dx\,ds,$$
 for all $H\in L^2([0,T],\bb D_W)$.
 
From the uniqueness of the generalized weak derivative, we have that $G_j(s,x) = - \partial_{W_j} u(s,x)$.

Now, suppose $u$ belongs to $L^2([0,T],H_{1,W}(\bb T^d))$ and let $H\in L^2([0,T],\bb D_W)$. Then, we have
\begin{align*}
2\int_0^T\, \int_{\bb T^d}  (\partial_{x_j}\partial_{W_j}H) (s, x) \, u(s,x)dx\, ds 
-  \int_0^T ds\, \int_{\bb T^d} &\left(\partial_{W_j}H (s, x)\right)^2  d(x^j\x W_j)=\\
-2\int_0^T\int_{\bb T^d} \partial_{W_j} H(s,x)\partial_{W_j} u(s,x) d(x^j\x W_j) 
-\int_0^T\int_{\bb T^d} &\left(\partial_{W_j} H(s,x)\right)^2 d(x^j\x W_j)
\end{align*}
We can rewrite the right-hand side of the above expression as
\begin{equation}\label{energia2}
-2\<\partial_{W_j} H, 2\partial_{W_j}u + \partial_{W_j} H\>_{x^j\x W_j}.
\end{equation}

A simple calculation shows that, for a Hilbert space $\mc H$ with inner product $<\!\!\cdot,\cdot\!\!>$, the following inequality holds:
$$ -<\!\!v,u+v\!\!> \;\leq\; \frac{1}{4}<\!\!u,u\!\!>,$$
for all $u,v\in\mc H$, and we have equality only when $v = -1/2 u$. 

Therefore, by the previous estimates and \eqref{energia2}
\begin{align*}
2\int_0^T\, \int_{\bb T^d}  (\partial_{x_j}\partial_{W_j}H) (s, x) \, u(s,x)dx\, ds 
-  \int_0^T ds\, \int_{\bb T^d} &\left(\partial_{W_j}H (s, x)\right)^2  d(x^j\x W_j) \leq\\ 
\int_0^T \int_{\bb T^d} \left(\partial_{W_j} u(s,x)\right)^2 d(x^j\x W_j).
\end{align*}
By the definition of energy, we have for each $j=1,\ldots,d,$
$$\mc Q_j(u) \leq \int_0^T \int_{\bb T^d} \left(\partial_{W_j} u(s,x)\right)^2 d(x^j\x W_j).$$
Hence, the total energy is finite. Using the fact that $L^2([0,T],\bb D_W)$ is dense in $L^2([0,T],H_{1,W}(\bb T^d))$, we have that
\begin{eqnarray*}
\mc Q(u) &=& \sum_{j=1} \int_0^T \|\partial_{W_j} u\|_{x^j\x W_j}^2 dt\\
&=& \int_0^T \|\nabla_W u\|_{L^2_W(\bb T^d)}^2dt.
\end{eqnarray*}
\end{proof}

\subsection{Uniqueness of weak solutions of the parabolic equation}\label{unicidade}

Recall that we denote by $\<\cdot, \cdot\>$ the inner product of the
Hilbert space $L^2(\bb T^d)$. Fix $H,G\in L^2(\bb T^d)$, $\lambda >0$, and denote by $H_\lambda$ and $G_\lambda$ in $H_{1,W}(\bb T^d)$ the unique weak solutions of the elliptic equations
$$\lambda H_\lambda - \nabla A\nabla_W H_\lambda = H,$$
and
$$\lambda G_\lambda - \nabla A\nabla_W G_\lambda = G,$$
respectively. Then, we have the following symmetry property
$$\<G_\lambda,H\> = \< G,H_\lambda\>.$$
In fact, both terms in the previous equality are equal to
$$\lambda \int_{\bb T^d} H_\lambda G_\lambda + \sum_{j=1}^d a_{jj} \int_{\bb T^d} (\partial_{W_j} H_\lambda)(\partial_{W_j}G_\lambda) d(x^j\x W_j).$$

Let $\rho: \bb R_+ \times \bb T \to [l,r]$ be a weak solution of the parabolic equation
\eqref{parabolic}. Since $\rho$, $\Phi(\rho) \in L^2([0,T],H_{1,W}(\bb T^d))$, and the set $\bb D_W$ is dense in $H_{1,W}(\bb T^d)$, we have for every $H$ in $H_{1,W}(\bb T^d)$,
\begin{equation}\label{solfraca}
\< \rho_t, H\> \;-\; \< \gamma , H\> 
\;=-\; \sum_{j=1}^d a_{jj} \int_0^t \< \partial_{W_j} \Phi(\rho_s) ,\partial_{W_j} H \>_{x^j\x W_j}\, ds
\end{equation}
for all $t>0$.

Denote by $\rho_s^\lambda\in H_{1,W}(\bb T^d)$ the unique weak solution of the elliptic equation
\begin{equation}\label{unici-elip}
\lambda \rho_s^\lambda - \nabla A\nabla_W \rho_s^\lambda = \rho(s,\cdot).
\end{equation}

We claim that 
\begin{equation}
\label{f08}
\< \rho_t \,,\, \rho_t^\lambda\> \;-\; \< \rho_0 \,,\,\rho_0^\lambda\> \;=\; -2\sum_{j=1}^d a_{jj} \int_{0}^{t} \< \partial_{W_j}\Phi(\rho_{s}) \,,\, \partial_{W_j}\rho_{s}^\lambda\>_{x^j\x W_j}\, ds
\end{equation}
for all $t>0$.

To prove this claim, fix $t>0$ and consider a partition
$0=t_0< t_1 < \cdots < t_n = t$ of the interval $[0,t]$. Using the telescopic sum, we obtain
\begin{eqnarray*}
\< \rho_t \,,\, \rho_t^\lambda\> \;-\; \< \rho_0 \,,\, \rho_0^\lambda\> & =& \sum_{k=0}^{n-1} \< \rho_{t_{k+1}} \,,\, 
\rho_{t_{k+1}}^\lambda\> \;-\; \< \rho_{t_{k+1}} \,,\, 
\rho_{t_{k}}^\lambda\> \\
&+& \sum_{k=0}^{n-1} \< \rho_{t_{k+1}} \,,\, \rho_{t_{k}}^\lambda\>
\;-\; \< \rho_{t_{k}} \,,\, \rho_{t_{k}}^\lambda\>\;.
\end{eqnarray*}

We handle the first term, the second one being similar. From the symmetric property of 
the weak solutions, $\rho_{t_{k+1}}^\lambda$
belongs to $H_{1,W}(\bb T^d)$ and since $\rho$ is a weak solution of
\eqref{parabolic},
\begin{equation*}
\< \rho_{t_{k+1}} \,,\, \rho_{t_{k+1}}^\lambda\> \;-\; 
\< \rho_{t_{k+1}} \,,\, \rho_{t_{k}}^\lambda\> \;=\;
-\sum_{j=1}^d a_{jj} \int_{t_{k}}^{t_{k+1}} \< \partial_{W_j}\Phi(\rho_{s}) \,,\, \partial_{W_j}
\rho_{t_{k+1}}^\lambda\>\, ds\;.
\end{equation*}
Add and subtract $\< \partial_{W_j}\Phi(\rho_{s}) \,,\, \partial_{W_j}\rho_{s}^\lambda\>$ inside the integral on the right hand side of the above expression. The time integral of this term is exactly
the expression announced in \eqref{f08} and the remainder is given by
\begin{equation*}
\sum_{j=1}^d a_{jj}\int_{t_{k}}^{t_{k+1}} \Big\{\< \partial_{W_j}\Phi(\rho_{s}) \,,\, \partial_{W_j}
\rho_{s}^\lambda\>  \;-\;\< \partial_{W_j}\Phi(\rho_{s}) \,,\, \partial_{W_j}
\rho_{t_{k+1}}^\lambda\>  \Big\} \, ds\;.
\end{equation*}

Since $\rho_s^\lambda$ is the unique weak solution of the elliptic equation \eqref{unici-elip}, and the weak solution has the symmetric property, we may rewrite the previous difference as
\begin{equation*}
 \, \Big\{ \< \Phi(\rho_{s}) \,,\, \rho_{t_{k+1}}\> \;-\; 
\< \Phi(\rho_{s}) \,,\, \rho_{s}\>  \Big\} \; 
-\; \lambda \Big\{ \< \Phi(\rho_{s})^\lambda \,,\,  
\rho_{t_{k+1}}\> \;-\; \< \Phi(\rho_{s})^\lambda \,,\, \rho_{s}\>
\Big\}\;.
\end{equation*}
The time integral between $t_k$ and $t_{k+1}$ of the second term is
equal to
\begin{equation*}
-\lambda \int_{t_{k}}^{t_{k+1}} ds \int_{s}^{t_{k+1}}  
\< \partial_{W_j} \Phi(\rho_{s})^\lambda \,,\, \partial_{W_j}\Phi(\rho_{r})\> \; dr
\end{equation*}
because $\rho$ is a weak solution of \eqref{parabolic} and $\Phi(\rho_{s})$
belongs to $H_{1,W}(\bb T^d)$. It follows from the boundedness of the weak solution given in Proposition \ref{lm-lambda} and from the boundedness of the $L^2_{x^j\x W_j}(\bb T^d)$ norm of $\partial_{W_j}\Phi (\rho)$ obtained in expression \eqref{solfraca}, that this expression is of order $(t_{k+1} - t_{k})^2$.

To conclude the proof of claim \eqref{f08} it remains to be shown that
\begin{equation*}
\sum_{k=0}^{n-1} \int_{t_{k}}^{t_{k+1}} \Big\{ \< \Phi(\rho_{s}) \,
,\, \rho_{t_{k+1}}\> \;-\;  \< \Phi(\rho_{s}) \,,\, \rho_{s}\>  \Big\}
\, ds
\end{equation*}
vanishes as the mesh of the partition tends to $0$. Using, again, the fact that $\rho$ is a weak
solution, we may rewrite the sum as
\begin{equation*}
-\sum_{k=0}^{n-1} \int_{t_{k}}^{t_{k+1}} ds \int_s^{t_{k+1}}
\< \partial_{W_j}\Phi(\rho_{s}) \,,\, \partial_{W_j}\Phi(\rho_{r}) \> 
\; dr\;.
\end{equation*}

We have that this expression vanishes as the mesh of the
partition tends to $0$ from the boundedness of the $L^2_{x^j\x W_j}(\bb T^d)$ norm of $\partial_{W_j}\Phi (\rho)$. This proves \eqref{f08}.

Recall the definition of the constant $B$ given at the beginning of this Section.

\begin{lemma}
\label{s12}
Fix $\lambda>0$, two density profiles $\gamma^1$, $\gamma^2:\bb T\to [l,r]$ and
denote by $\rho^1$, $\rho^2$ weak solutions of \eqref{parabolic} with
initial value $\gamma^1$, $\gamma^2$, respectively. Then, 
\begin{equation*}
\Big \< \rho^1_t - \rho^2_t \,,\,
\rho^{1,\lambda}_t - \rho^{2,\lambda}_t \Big \> \;\le\; 
\Big \< \gamma^1 - \gamma^2 \,,\, \gamma^{1,\lambda} - \gamma^{2,\lambda} \Big \> \, e^{ B \lambda t/2}
\end{equation*}
for all $t>0$.  In particular, there exists at most one
weak solution of \eqref{parabolic}.
\end{lemma}

\begin{proof}
We begin by showing that if there exists $\lambda>0$ such that
$$\<H,H^\lambda\> = 0,$$
then $H=0$. In fact, we would have the following 
\begin{eqnarray*}
\int_{\bb T^d}\lambda (H^\lambda)^2 dx + \sum_{j=1}^d a_{jj} \int_{\bb T^d} \left(\partial_{W_j}H^\lambda\right)^2 d(x^j\x W_j) = \int_{\bb T^d} H H^\lambda dx= 0,
\end{eqnarray*}
which implies that $\|H^\lambda\|_{H_{1,W}(\bb T^d)} = 0$, and hence $H_\lambda = 0$, which yields $H=0$.

Fix two density profiles $\gamma^1$, $\gamma^2:\bb T^d\to [l,r]$.  Let
$\rho^1$, $\rho^2$ be two weak solutions with initial values
$\gamma^1$, $\gamma^2$, respectively. By \eqref{f08}, for any
$\lambda>0$,
\begin{equation}\label{unic1}
\begin{array}{c}
\!\!\!\!\!\!\!\!\!\!\!\!\!\!\! 
\Big \< \rho^1_t - \rho^2_t \,,\,
\rho^{1,\lambda}_t - \rho^{2,\lambda}_t \Big \> - 
\Big \< \gamma^1 - \gamma^2 \,,\, \gamma^{1,\lambda} - \gamma^{2,\lambda} \Big \>\;= \\
\!\!\!\!\!\!\!\!\!\!\!\!\!\!\!
- 2 \int_{0}^{t} \< \Phi(\rho^1_{s}) - \Phi(\rho^2_{s}) \,,\, 
\rho^1_{s} -  \rho^2_{s} \>\, ds 
\;+\; 2 \lambda \int_{0}^{t} \Big \< \Phi(\rho^1_{s}) - \Phi(\rho^2_{s}) 
\,,\,   \rho^{1,\lambda}_{s} -  \rho^{2,\lambda}_{s} \Big \>\, ds\;.
\end{array}
\end{equation}

Define the inner product in $H_{1,W}(\bb T^d)$
$$\<u,v\>_{\lambda} = \<u,v^\lambda\>.$$
This is, in fact, an inner product, since $\<u,v\>_\lambda = \<v,u\>_\lambda$ by the symmetric property, and if $u\neq 0$, then $\<u,u\>_\lambda >0$:
$$\int_{\bb T^d} u u_\lambda dx = \lambda \int_{\bb T^d} u_\lambda^2 dx + \sum_{j=1}^d a_{jj}\int_{\bb T^d} \left(\partial_{W_j} u^\lambda\right)^2 d(x^j\x W_j).$$
The linearity of this inner product can be easily verified.

Then, we have
$$2 \lambda \int_{0}^{t} \Big \< \Phi(\rho^1_{s}) - \Phi(\rho^2_{s}) 
\,,\,   \rho^{1,\lambda}_{s} -  \rho^{2,\lambda}_{s} \Big \>\, ds = 2\lambda\int_0^t \Big \< \Phi(\rho^1_{s}) - \Phi(\rho^2_{s}) 
\,,\,   \rho^{1}_{s} -  \rho^{2}_{s} \Big \>_\lambda ds.$$

By using the Cauchy-Schwartz inequality twice, the term on the right hand side of the above formula is
bounded above by
\begin{equation*}
\frac 1A  \int_{0}^{t} \Big \< \Phi(\rho^1_{s}) - \Phi(\rho^2_{s}) 
\,,\,  \Phi(\rho^{1}_{s})^\lambda - \Phi(\rho^{2}_{s})^\lambda \Big \>\, ds
\;+\; A\lambda^2  \int_{0}^{t} \Big \< \rho^1_{s} -  \rho^2_{s}
\,,\, \rho^{1,\lambda}_{s} -  \rho^{2,\lambda}_{s}\Big \>\, ds
\end{equation*}
for every $A>0$. From Proposition \ref{lm-lambda}, we have that $\|u^\lambda\|\leq \lambda^{-1}\|u\|$, and since $\Phi'$ is bounded by $B$, the first term of
the previous expression is less than or equal to
\begin{equation*}
\frac B{A \lambda} \int_{0}^{t} \Big \< \rho^1_{s} - \rho^2_{s} 
\,,\,    \Phi(\rho^1_{s}) - \Phi(\rho^2_{s}) \Big \>\, ds \;.
\end{equation*}
Choosing $A = B/2\lambda$, this expression cancels with the first term
on the right hand side of \eqref{unic1}. In particular, the left
hand side of this formula is bounded by
\begin{equation*}
\frac {B \lambda}2 \int_{0}^{t} \Big \< \rho^1_{s} -  \rho^2_{s}
\,,\,   \rho^{1,\lambda}_{s} -  \rho^{2,\lambda}_{s} \Big \>\, ds\;.
\end{equation*}
To conclude, recall Gronwall's inequality.
\end{proof}

\begin{remark}
Let $\mc L_W^A:\mc D_W\to L^2(\bb T^d)$ be the self-adjoint extension given in Remark \ref{a-adj}. For $\lambda>0$, define the resolvent operator $G_\lambda^A = (\lambda\bb I - \mc L_W^A)^{-1}$. Following \cite{TC,v}, another possible definition of weak solution of equation \eqref{parabolic} is given as follows: a bounded function $\rho : [0,T] \times \bb T^d \to [l,r]$
is said to be a weak solution of the parabolic differential equation \eqref{parabolic} if
\begin{equation}
\label{wsol}
\< \rho_t, G_\lambda^A h\> \;-\; \< \gamma , G_\lambda^A h\>
\;=\; \int_0^t \< \Phi(\rho_s) , \mc L_W^A G_\lambda^A h \>\, ds\;
\end{equation}
for every continuous function $h:\bb T^d\to \bb R$, $t\in[0,T]$, and all $\lambda >0$. We claim that this definition of weak solution coincides with our definition introduced at the beginning of Section \ref{sec4}. Indeed, for continuous $h:\bb T^d\to\bb R$, $G_\lambda^A h$ belongs to $\mc D_W$. Since $\bb D_W$ is dense in $\mc D_W$ with respect to the $H_{1,W}(\bb T^d)$-norm, it follows that our definition implies the current definition. Conversely, since the set of continuous functions is dense in $L^2(\bb T^d)$, the identity \eqref{wsol} is valid for all $h\in L^2(\bb T^d)$. Therefore, for each $H \in \mc D_W$ we have
\begin{equation*}
\< \rho_t, H\> \;-\; \< \gamma , H\>
\;=\; \int_0^t \< \Phi(\rho_s) , \mc L_W^A H \>\, ds.
\end{equation*}

In particular, the above identity holds for every $H\in \bb D_W$, and therefore the integral identity in our definition of weak solutions holds.

It remains to be checked that the weak solution of the current definition belongs to $L^2([0,T],H_{1,W}(\bb T^d))$. This follows from the fact that there exists at most one weak solution satisfying \eqref{wsol}, that this unique solution has finite energy, and from Lemma \ref{energia}. A proof of the fact that there exists at most one solution satisfying \eqref{wsol}, and that this unique solution has finite energy, can be found in \cite{TC,v}.

Finally, the integral identity of our definition of weak solution has an advantage regarding the integral identity \eqref{wsol}, due to the fact that we do not need the resolvent operator $G_\lambda^A$ for any $\lambda$. Moreover, we have an explicit characterization of our test functions.
\end{remark}

\section{$W$-Generalized Sobolev spaces: Discrete version}
\label{sec5}
We will now establish some of the results obtained in the above sections to the discrete version of the $W$-Sobolev space. Our motivation to obtain these results is that they will be useful when studying homogenization in Section \ref{sec6}. We begin by introducing some definitions and notations.

Fix $W$ as in \eqref{w} and functions $f,g$ defined on $N^{-1}\bb T^d_N$. Consider the following difference operators: $\partial^N_{x_j}$, which is the standard difference operator, 
\begin{equation*}
\partial^N_{x_j}f\left(\frac xN\right)\; =\; N\left[f\left(\frac{x+e_j}{N}\right)-f\left(\frac xN\right)\right]\;,
\end{equation*}
and $\partial^N_{W_j}$, which is the $W_j$-difference operator:
\begin{equation*}
\partial^N_{W_j}f\left(\frac xN\right)\;=\; \frac{f\left(\frac{x+e_j}{N}\right) -f\left(\frac xN\right)}{W\left(\frac{x+e_j}{N}\right) -W\left(\frac xN\right)},
\end{equation*}
for $x\in\bb T^d_N$.
We introduce the following scalar product 
\begin{align*}
 \<f,g\>_N& \;:=\; \frac1{N^d}\sum_{x\in \bb T^d_N} f(x)g(x),\\
\<f,g\>_{W_j,N} :=& \frac1{N^{d-1}} \sum_{x\in\bb T^d_N}f(x)g(x)\big(W((x+e_j)/N)-W(x/N)\big),\\
\<f,g\>_{1,W,N}\;:=\; &\<f,g\>_N\; +\; \sum_{j=1}^d \<\partial_{W_j}^Nf,\partial_{W_j}^Ng\>_{W_j,N},
\end{align*}
and its induced norms 
$$\| f \|^2_{L^2(\bb T^d_N)} = \<f,f\>_N,\quad \| f \|^2_{L_{W_j}^2(\bb T^d_N)} = \<f,f\>_{W_j,N}\hbox{~and~}\| f \|^2_{H_{1,W}(\bb T^d_N)} = \<f,f\>_{1,W,N}.$$

These norms are natural discretizations of the norms introduced in the previous sections. Note that the properties of the Lebesgue's measure used in the proof of Corollary \ref{poinc}, also holds for the normalized counting measure. Therefore, we may use the same arguments of this Corollary to prove its discrete version.

\begin{lemma}[Discrete Poincaré Inequality]
There exists a finite constant $C$ such that
$$\left\|f-\frac{1}{N^d}\sum_{x\in \bb T^d}f \right\|_{L^2(\bb T^d_N)} \;\le\;C\|\nabla_W^Nf\|_{L^2_W({\bb T}_N^d)},$$ 
where
$$\|\nabla_Wf\|_{L^2_W({\bb T}_N^d)}^2 = \sum_{j=1}^d \|\partial_{W_j}^N f\|_{L^2_{W_j}(\bb T_N^d)}^2,$$
for all $f:N^{-1}\bb T_N^d\to\bb R$.
\end{lemma}

Let $A$ be a diagonal matrix satisfying \eqref{cota A}. We are interested in studying the problem
\begin{equation}\label{prob TN}
T_{\lambda}^Nu = f,
\end{equation}
where $u:N^{-1}\bb T^d_N \to \bb R$ is the unknown function, $f:N^{-1}\bb T^d_N \to \bb R$ is given, and $T_{\lambda}^N$ denotes the discrete generalized elliptic operator
\begin{equation}\label{def TN}
  T_{\lambda}^Nu\; :=\;\lambda u - \nabla^N A\nabla_W^Nu,
\end{equation}
with
$$\nabla^N A\nabla_W^Nu\;:=\;\sum_{i=1}^d\partial_{x_i}^N\Big(a_i(x/N)\partial_{W_i}^Nu\Big).$$

The bilinear form $B^N[\cdot,\cdot]$ associated with the elliptic operator $T_{\lambda}^N$ is given by
\begin{equation}\label{def BN}
\begin{array}{c}
B^N[u,v] \;=\; \lambda\<u,v\>_N\;+\\*[5pt]
+\frac{1}{N^{d-1}}\sum_{i=1}^d\sum_{x\in\bb T_N^d}a_i(x/N)(\partial_{W_i}^N u)(\partial_{W_i}^Nv)[W_i((x_i+1)/N)-W_i(x_i/N)],
\end{array}
\end{equation}
where $u,v:N^{-1}\bb T_N^d\to\bb R$.

A function $u:N^{-1}\bb T_N^d\to\bb R$ is said to be a weak solution of the equation
$T_\lambda^N u = f\;$ if 
\begin{equation*}
B^N[u,v]\;=\;\<f,v\>_N \;\;\text{for all}\;\;v:N^{-1}\bb T_N^d\to\bb R.
\end{equation*}

We say that a function $f:N^{-1}\bb T^d_N \to \bb R$ belongs to the discrete space of functions orthogonal to the constant functions $H^{\bot}_N(\bb T^d_N)$ if 
$$\frac1{N^d}\sum_{x\in \bb T^d_N} f(x/N) = 0.$$

The following results are analogous to the weak solutions of generalized elliptic equations for this discrete version. We remark that the proofs of these lemmas are identical to the ones in the continuous case. Furthermore, the weak solution for the case $\lambda=0$ is unique in $H^{\bot}_N(\bb T^d_N)$.

\begin{lemma}
The equation
$$\nabla^N A\nabla_W^N u = f,$$
has weak solution $u:N^{-1}\bb T_N^d\to\bb R$ if and only if
$$\frac{1}{N^d}\sum_{x\in{\bb T}_N^d} f(x) = 0.$$
In this case we have uniqueness of the solution disregarding addition by constants. Moreover, if $u\in H^{\bot}_N(\bb T^d_N)$ we have the bound
$$\| u \|_{{H}_{1,W}(\bb T_N^d)} \leq C \|f\|_{L^2(\bb T_N^d)},\hbox{~and~}\; \| u \|_{L^2(\bb T_N^d)} \leq \lambda^{-1} \|f\|_{L^2(\bb T_N^d)},$$
where $C>0$ does not depend on $f$ nor $N$.
\end{lemma}

\begin{lemma}\label{lmdiscreto}
Let $\lambda > 0$. There exists a unique weak solution $u:N^{-1}\bb T_N^d\to\bb R$ of the equation
\begin{equation}\label{eqdisc}
\lambda u - \nabla^N A \nabla_W^N u = f.
\end{equation}
Moreover,
$$\| u \|_{{H}_{1,W}(\bb T_N^d)} \leq C \|f\|_{L^2(\bb T_N^d)},\hbox{~and~}\; \| u \|_{L^2(\bb T_N^d)} \leq \lambda^{-1} \|f\|_{L^2(\bb T_N^d)},$$
where $C>0$ does not depend neither on $f$ nor $N$.
\end{lemma}
\begin{remark}\label{diracdiscreta}
Note that in the set of functions in $\bb T_N^d$ we have a ``Dirac measure'' concentrated in a point $x$ as a function: the function that takes value $N^d$ in $x$ and zero elsewhere. Therefore, we may integrate these weak solutions with respect to this function to obtain that every weak solution is, in fact, a strong solution.
\end{remark}
\subsection{Connections between the discrete and continuous Sobolev spaces}\label{correspondencia}
Given a function $f\in H_{1,W}(\bb T^d)$, we can define its restriction $f_N$ to the lattice $N^{-1}\bb T_N^d$ as
$$f_N(x) = f(x) \;\;\text{if}\;\; x\in N^{-1}\bb T_N^d.$$

However, given a function $f:N^{-1}\bb T_N^d\to\bb R$ it is not straightforward how to define an extension belonging to $H_{1,W}(\bb T^d)$. To do so, we need the definition of $W$-interpolation, which we give below.

Let $f_N:N^{-1}\bb T_N\to \bb R$ and $W: \bb R \to \bb R$, a \emph{strictly increasing} right continuous function with left limits (c\`adl\`ag), and periodic. 
The $W$-interpolation $f_N^*$ of $f_N$ is given by:
\begin{eqnarray*}
f^*_N(x+t)&:=&\frac{W((x+1)/N)-W((x+t)/N)}{W((x+1)/N)-W(x/N)}f(x)+\\
&+&\frac{W((x+t)/N)-W(x/N)}{W((x+1)/N)-W(x/N)}f(x+1)
\end{eqnarray*}
for $0\le  t < 1$. 
Note that
\begin{equation*}
\frac{\partial f^*_N}{\partial W}(x + t) = \frac{f(x+1)-f(x)}{W((x+1)/N)-W(x/N)}\;=\;\partial^N_Wf(x).
\end{equation*}

Using the standard construction of $d$-dimensional linear interpolation, it is possible to define the $W$-interpolation of a function $f_N:\bb T_N^d\to \bb R$, with $W(x) = \sum_{i=1}^d W_i(x_i)$ as defined in \eqref{w}. 

We now establish the connection between the discrete and continuous Sobolev spaces by showing how a sequence of functions defined in $\bb T_N^d$ can converge to a function in $H_{1,W}(\bb T^d)$.

We say that a family $f_N\in L^2(\bb T^d_N)$ converges strongly (resp. weakly) to the function $f\in L^2(\bb T^d)$
as $N\to \infty$ if $f^*_N$ converges strongly (resp. weakly) to the function $f$. From now on we will omit the symbol ``\,\,${}^*$\,\,'' in the $W$-interpolated function, and denoting them simply by $f_N$.

The convergence in $H^{-1}_W(\bb T^d)$ can be defined in terms of duality. Namely, we say that a functional $f_N$ on $\bb T_N^d$ converges to $f\in H^{-1}_W(\bb T^d)$ strongly (resp. weakly) if for any sequence of functions $u_N:\bb T_N^d\to\bb R$ and $u\in H_{1,W}(\bb T^d)$ such that $u_N\to u$ weakly (resp. strongly) in $H_{1,W}(\bb T^d)$, we have
$$(f_N,u_N)_N \longrightarrow (f,u), \quad \hbox{as~} N\to\infty.$$

\begin{remark}\label{cotalm}
Suppose in Lemma \ref{lmdiscreto} that $f\in L^2(\bb T^d)$, and let $u$ be a weak solution of the problem \eqref{eqdisc}, then we have the following bound
$$\|u\|_{H_{1,W}(\bb T_N^d)} \leq C \|f\|_{L^2(\bb T^d)},$$
since $\|f\|_{L^2(\bb T_N^d)} \to \|f\|_{L^2(\bb T^d)}$ as $N\to\infty$. 
\end{remark}

\section{Homogenization}
\label{sec6}
In this ``brief'' Section we prove a homogenization result for the $W$-generalized differential operator. We follow the approach considered in \cite{pr}. The study of homogenization is motivated by several applications in mechanics, physics, chemistry and engineering. The focus of our approach is to study the asymptotic behavior of effective coefficients for a family of random difference schemes whose coefficients can be obtained by the discretization of random high-contrast lattice structures.

This Section is structured as follows: in subsection 6.1 we define the concept of $H$-convergence together with some properties; subsection 6.2 deals with a description of the random environment along with some definitions, whereas the main result is proved in subsection 6.3.

\subsection{$H$-convergence}
We say that the diagonal matrix $A^N=(a_{jj}^N)$ $H$-converges to the diagonal matrix $A = (a_{jj})$, denoted by $A^N\stackrel{H}{\longrightarrow} A$, if, for every sequence $f^N\in H^{-1}_W (\bb T_N^d)$ such that $f^N\to f$ as $N\to\infty$ in $H^{-1}_W (\bb T^d)$, we have
\begin{itemize}
\item $u_N \to u_0$ weakly in $H_{1,W}(\bb T^d)$ as $N\to\infty$,
\item $a_{jj}^N \partial_{W_j}^N u_N\to a_{jj} \partial_{W_j} u_0$ weakly in $L^2_{x^j\x W_j}(\bb T^d)$ for each $j=1,\ldots,d$,
\end{itemize}
where $u_N:\bb T_N^d\to\bb R$ is the solution of the problem
$$\lambda u_N - \nabla^N A^N \nabla_W^N u_N = f_N,$$
and $u_0\in H_{1,W}(\bb T^d)$ is the solution of the problem
$$\lambda u_0 - \nabla A\nabla_W u_0 = f.$$
The notion of convergence used in both items above was defined in subsection \ref{correspondencia}.

We now obtain a property regarding $H$-convergence.

\begin{proposition}\label{hconvergencia}
Let $A^N\stackrel{H}{\longrightarrow} A$, as $N\to\infty$, with $u_N$ being the solution of
$$\lambda u_N - \nabla^N A^N \nabla_W^N u_N = f,$$
where $f\in H_W^{-1}(\bb T^d)$ is fixed. Then, the following limit relations hold true:
$$\frac{1}{N^d}\sum_{x\in\bb T_N^d} u_N^2(x) \to \int_{\bb T^d} u_0^2(x) dx,$$
and
\begin{align*}
\frac{1}{N^{d-1}}\sum_{j=1}^d\sum_{x\in\bb T_N^d} a_{jj}^N(x) (\partial_{W_j}^N u_N(x))^2 &\left[W_j((x_j+1)/N)-W_j(x_j/N)\right]\\
 &\to \sum_{j=1}^d \int_{\bb T^d} a_{jj}(x) (\partial_{W_j} u_0(x))^2 d(x^j\x W_j),
\end{align*}
as $N\to\infty$.
\end{proposition}
\begin{proof}
We begin by noting that
\begin{equation}\label{hconv1}
\frac{1}{N^d}\sum_{x\in\bb T_N^d} f (u_N-u_0)\to 0,
\end{equation}
as $N\to\infty$ since $u_N-u_0$ converges weakly to 0 in $H_{1,W}(\bb T^d)$. On the other hand, we have
\begin{eqnarray*}
\frac{1}{N^d}\sum_{x\in\bb T_N^d} f (u_N-u_0) & = & \frac{1}{N^d}\sum_{x\in\bb T_N^d} (\lambda u_N -\nabla^N A^N\nabla_W^N u_N) (u_N-u_0)\\
&=&  \frac{\lambda}{N^d}\sum_{x\in\bb T_N^d} u_N^2 - \frac{1}{N^d}\sum_{x\in\bb T_N^d} u_N \nabla^N A^N\nabla_W^N u_N\\
&-&\frac{\lambda}{N^d}\sum_{x\in\bb T_N^d} u_N u_0+ \frac{1}{N^d}\sum_{x\in\bb T_N^d}u_0\nabla^N A^N\nabla_W^N u_N.
\end{eqnarray*}
Using the weak convergences of $u_N$ and $a_{jj}\partial_{W_j}^N u_N$, and the convergence in \eqref{hconv1}, we obtain, after a summation by parts in the above expressions,
\begin{align}
\frac{\lambda}{N^d}\sum_{x\in\bb T_N^d} u_N^2 + \frac{1}{N^{d-1}}\sum_{j=1}^d &\sum_{x\in\bb T_N^d} a_{jj}^N (\partial_{W_j}^N u_N)^2 [W_j((x_j+1)/N)-W_j(x_j)]\nonumber\\
&\stackrel{N\to\infty}{\longrightarrow} \lambda\int_{\bb T^d} u_0^2 dx + \sum_{j=1}^d \int_{\bb T^d} a_{jj} (\partial_{W_j} u_0)^2 d(x^j\x W_j).\label{hconv2}
\end{align}
By Lemma \ref{lmdiscreto}, the sequence $u_N$ is $\|\cdot\|_{1,W}$ bounded 	
uniformly. Suppose, now, that $u_N$ does not converge to $u_0$ in $L^2(\bb T^d)$. That is, there exist $\epsilon>0$ and a subsequence $(u_{N_k})$ such that
$$\|u_{N_k} - u_0\|_{L^2(\bb T^d)}>\epsilon,$$
for all $k$. By Rellich-Kondrachov Theorem (Proposition \ref{f3}), we have that there exists $v\in L^2(\bb T^d)$ and a further subsequence (also denoted by $u_{N_k}$) such that
$$u_{N_k} \stackrel{k\to\infty}{\longrightarrow} v,\quad\hbox{~in~} L^2(\bb T^d).$$
This implies that
$$u_{N_k} \to v, \quad \hbox{weakly in~}L^2(\bb T^d),$$
but this is a contradiction, since
$$u_{N_k} \to u_0,\quad \hbox{weakly in~}L^2(\bb T^d),$$
and $\|v-u_0\|_{L^2(\bb T^d)} \geq \epsilon.$ Therefore, $u_N \to u_0$ in $L^2(\bb T^d)$. The proof thus follows from expression \eqref{hconv2}.
\end{proof}
This Proposition shows that even though the $H$-convergence only requires weak convergence in its definition, it yields a convergence in the strong sense (convergence in the $L^2$-norm).
\subsection{Random environment}
In this subsection we introduce the statistically homogeneous rapidly oscillating coefficients that will be used to define the random $W$-generalized difference elliptic operators, where the $W$-generalized difference elliptic operator was given in Section \ref{sec5}. 

Let $(\Omega,\mc F, \mu)$ be a standard probability space and $\{ T_x : \Omega \to \Omega; x\in \bb Z^d\}$
be a group of $\mc F$-measurable and ergodic transformations which preserve the measure $\mu$:
\begin{itemize}
\item $ T_x : \Omega \to \Omega$ is $\mc F$-measurable for all $x \in \bb Z^d$,
\item $\mu(T_x \textbf{A}) = \mu(\textbf{A})$, for any $\textbf{A} \in \mc F$ and $x\in \bb Z^d$,
\item $T_0 = \textit{I}\;,\; \;T_x\circ T_y = T_{x+y}$,
\item For any  $f\in L^1(\Omega)$ such that $f(T_x\omega)=f(\omega)\;\; \mu$-a.s  for each $x\in \bb Z^d$, 
is equal to a constant $\mu$-a.s.
\end{itemize}
The last condition implies that the group $T_x$ is ergodic. 

Let us now introduce the vector-valued $\mc F$-measurable functions $\{a_j(\omega) ; j=1,\ldots, d\}$ such that there exists
$\theta >0$ with $$\theta^{-1} \le a_j(w)\le \theta,$$ 
for all $\omega \in \Omega$ and $ j= 1,\ldots, d$.  Then, define the diagonal matrices $A^N$ whose elements  are given by
\begin{equation}
\label{AN}
a^N_{jj}(x):=a^N_j = a_j(T_{Nx}\omega)\;,\;\;x\in T^d_N\;,\;\;j = 1, \ldots , d.
\end{equation}
\subsection{Homogenization of random operators}

Let $\lambda>0$, $f_N$ be a functional on the space of functions $h_N:\bb T_N^d\to\bb R$, $f\in  H^{-1}_W(\bb T^d)$ (see also, subsection \ref{dual2}), $u_N$ be the unique weak solution 
of $$\lambda u_N - \nabla^N A^N \nabla_W^N u_N = f_N,$$ 
and $u_0$ be the unique weak solution of
\begin{equation}\label{h4}
\lambda u_0 - \nabla A\nabla_W u_0 = f.
\end{equation}
For more details on existence and uniqueness of such solutions see Sections \ref{sec3} and \ref{sec5}.

We say that the diagonal matrix $A$ is a \textit{homogenization} of the sequence of random matrices $A^N$ if the following conditions hold:
\begin{itemize}
\item For each  sequence $f_N\to f$ in $H^{-1}_W({\bb T^d})$, $u_N$ converges weakly in $ H_{1,W}$ to $u_0$, when $N\to\infty$;\\
\item $a_i^N \partial_{W_i}^N u^N \to a_i \partial_{W_i} u,$ weakly in $L^2_{x^i\x W_i}({\bb T}^d)$ when $N\to\infty$.
\end{itemize}

Note that homogenization is a particular case of $H$-convergence. 

We will now state and prove the main result of this Section.

\begin{theorem}\label{homoge}
Let $A^N$ be a sequence of ergodic random matrices, such as the one that defines our random environment. Then, almost surely, $A^N(\omega)$ admits a homogenization, where the homogenized matrix $A$ does not depend on the realization $\omega$.
\end{theorem}
\begin{proof}
Fix $f\in H^{-1}({\bb T}^d)$, and consider the problem
$$\lambda u_N - \nabla^N A^N \nabla_W^Nu_N = f.$$
Using Lemma \ref{lmdiscreto} and Remark \ref{cotalm}, there exists a unique weak solution $u_N$ of the problem above, such that its $H_{1,W}^N$ norm is uniformly bounded in $N$. That is, there exists a constant $C>0$ such that
$$\|u_N\|_{H_{1,W}(\bb T_N^d)} \leq C \|f\|_{L^2(\bb T^d)}.$$
Thus, the $L^2(\bb T_N^d)$-norm of $a_i^N \partial_{W_i}^N u_N$ is uniformly bounded.

From $W$-interpolation (see subsection \ref{correspondencia}) and the fact that $H_{1,W}(\bb T^d)$ is a Hilbert space (Lemma \ref{sobolevhilbert}), there exists a convergent subsequence of $u_N$ (which we will also denote by $u_N$) such that
$$u_N\to u_0, \qquad \hbox{weakly in}\quad {H}_{1,W}({\bb T}^d),$$
and
\begin{equation}\label{h1}
a_i^N \partial_{W_i}^N u_N\to v_0\qquad \hbox{weakly in}\quad L^2({\bb T}^d),
\end{equation}
as $N\to\infty$; $v_0$ being some function in $L^2_{x^i\x W_i}(\bb T^d)$.

First, observe that the weak convergence in $H_{1,W}({\bb T}^d)$ implies that 
\begin{equation}\label{h2}
\partial_{W_i}^N u_N\stackrel{N\to\infty}{\longrightarrow}\partial_{W_i} u\quad\hbox{weakly in}\quad L^2_{x^i\x W_i}(\bb T^d).
\end{equation}
From Birkhoff's ergodic theorem, we obtain the almost sure convergence, as $N$ tends to infinity, of the random coefficients:
\begin{equation}\label{h3}
a_i^N {\longrightarrow}\; a_i,
\end{equation}
where $a_i =  E[a_i^{N_0}]$, for any $N_0\in\bb N$.

From convergences in \eqref{h1}, \eqref{h2} and \eqref{h3}, we obtain that
$$v_0 = a_i \partial_{W_i} u_0,$$
where, from the weak convergences, $u_0$ clearly solves problem \eqref{h4}.

To conclude the proof it remains to be shown that we can pass from the subsequence to the sequence. This follows from uniquenesses of weak solutions of the problem \eqref{h4}.
\end{proof}

\begin{remark}
At first sight, one may think that we are dealing with a very special class of matrices $A$ (diagonal matrices). Nevertheless, the random environment for random walks proposed in \cite[Section 2.3]{pr}, which is also exactly the same random environment employed in \cite{gj}, results in diagonal matrices. This is essentially due to the fact that in symmetric nearest-neighbor interacting particle systems (for example, the zero-range dynamics considered in \cite{gj}), a particle at a site $x\in\bb T_N^d$ may jump to the sites $x\pm e_j$, $j=1,\ldots,d$. In such a case, the jump rate from $x$ to $x+e_j$ determines the $j$th element of the diagonal matrix.
\end{remark}

\begin{remark}
Note that if $u\in\bb D_W$ is a strong solution (or weak, in view of Remark \ref{diracdiscreta}) of
$$\lambda u - \nabla A \nabla_W u = f$$
and $u_N$ is strong solution of the discrete problem
$$\lambda u_N- \nabla^N A^N \nabla_W^Nu_N = f$$
then, the homogenization theorem also holds, that is, $u_N$ also converges weakly in $H_{1,W}$ to $u$.
\end{remark}

\section{Hydrodynamic limit of gradient processes with conductances in random environment}
\label{aplicacao-limite}
Lastly, as an application of all the theory developed in the previous sections, we prove a hydrodynamic limit for \textit{gradient processes with conductances in random environments}. Hydrodynamic limits for gradient processes with conductances have been obtained in \cite{TC} for the one-dimensional setup and in \cite{v} for the $d$-dimensional setup. However, the proof given here is much simpler and more natural, in view of the theory developed here, than the proofs given in \cite{TC,v}. Furthermore, the proof of this hydrodynamic limit also provides an existence theorem for the $W$-generalized parabolic equations in \eqref{parabolic}.

The hydrodynamic limit allows one to deduce the macroscopic behavior of the system from the microscopic interaction among particles. Moreover, this approach justifies rigorously a method often used by physicists to establish the partial differential equations that describe the evolution of the thermodynamic characteristics of a fluid.

This Section is structured as follows: in subsection 7.1 we present the model, derive some properties and fix the notations; subsection 7.2 deals with the hydrodynamic equation; finally, subsections 7.3 and 7.4 are devoted to the proof of the hydrodynamic limit. 
\subsection{The exclusion processes with conductances in random environments}

Fix a typical realization $\omega \in \Omega$ of the random environment defined in Section \ref{sec6}. For each $x\in \bb T^d_N$ and $j = 1,\ldots, d$, 
define the symmetric rate $\xi_{x, x+e_j}=\xi_{x+e_j, x}$ by
\begin{equation}\label{rate}
\xi_{x, x+e_j} \;=\; \frac{a^N_j(x)}{N[W((x+e_j)/N) - W(x/N)]}\;=\;
\frac{a^N_j(x)}{N[W_j((x_j+1)/N) - W_j(x_j/N)]}.
\end{equation}
where $a^N_j(x)$ is given by \eqref{AN}, and  ${e_1,\ldots ,e_d}$ is the canonical basis of $\bb R^d$. Also, let
$b> -1/2\;$ and
\begin{equation*}
c_{x,x+e_j}(\eta) \;=\; 1 \;+\; b \{ \eta(x-e_j) + \eta(x+2\ e_j)\}\;,
\end{equation*}
where all sums are modulo $N$.

Distribute particles on $\bb T^d_N$ in such a way that each site
of $\bb T^d_N$ is occupied at most by one particle. Denote by $\eta$ the
configurations of the state space $\{0,1\}^{\bb T^d_N}$ so that $\eta(x) =0$
if site $x$ is vacant, and $\eta(x)=1$ if site $x$ is occupied.

The  exclusion process with conductances in a random environment is a continuous-time Markov process
$\{\eta_t : t\ge 0\}$ with state space $\{0,1\}^{\bb T^d_N} =\{\eta:\bb T^d_N \to \{0,1\}\}$, 
whose generator $L_N$ acts on functions $f:
\{0,1\}^{\bb T^d_N} \to \bb R$ as
\begin{equation}
\label{g4}
(L_N f) (\eta) \;=\;\sum^d_{j=1} \sum_{x \in \bb T^d_N} \xi_{x,x+e_j}c_{x,x+e_j}(\eta)\, 
\{ f(\sigma^{x,x+e_j} \eta) - f(\eta) \} \;,
\end{equation}
where $\sigma^{x,x+e_j} \eta$ is the configuration obtained from $\eta$
by exchanging the variables $\eta(x)$ and $\eta(x+e_j)$:
\begin{equation}
\label{g5}
(\sigma^{x,x+e_j} \eta)(y) \;=\;
\begin{cases}
\eta (x+e_j) & \text{ if } y=x,\\
\eta (x) & \text{ if } y=x+e_j,\\
\eta (y) & \text{ otherwise}.
\end{cases}
\end{equation}

We consider the Markov process  $\{\eta_t : t\ge 0\}$ on the configurations $\{0,1\}^{\bb T^d_N}$ 
associated to the generator $L_N$ in the diffusive scale, i.e., $L_N$ is speeded up by $N^2$. 

We now describe the stochastic evolution of the process. After a time given by an exponential distribution, 	
a random choice of a point $x\in \bb T^d_N$ is made. At rate $\xi_{x,x+e_j}$ the occupation variables
$\eta(x)$, $\eta(x+e_j)$ are exchanged. Note that only nearest neighbor jumps are allowed. The conductances are given by the function $W$, whereas the random environment is given by the matrix $A^N:=(a_{jj}^N(x))_{d\times d}$. The discontinuity points of $W$ may, for instance, model a membrane which obstructs the passage of particles in a fluid. For more details  see \cite{v}.

The effect of the factor $c_{x,x+e_j}(\eta)$ is the following: if the parameter $b$ is positive, the presence of particles in the neighboring
sites of the bond $\{x,x+e_j\}$ speeds up the exchange rate by a factor of
order one, and if the parameter $b$ is negative, the presence of particles in the neighboring sites slows down the exchange rate also by a factor of order one. More details are given in Remark \ref{taxac} below.

The dynamics informally presented describes a Markov evolution.
 A computation shows that the Bernoulli product measures
$\{\nu^N_\alpha : 0\le \alpha \le 1\}$ are invariant, in fact
reversible, for the dynamics. The measure $\nu^N_\alpha$ is obtained
by placing a particle at each site, independently from the other
sites, with probability $\alpha$. Thus, $\nu^N_\alpha$ is a product
measure over $\{0,1\}^{\bb T^d_N}$ with marginals given by
\begin{equation*}
\nu^N_\alpha \{\eta : \eta(x) =1\} \;=\; \alpha
\end{equation*}
for $x$ in $\bb T^d_N$. For more details see \cite[chapter 2]{kl}.\\

Consider the random walk $\{X_t\}_{t\ge0}$ of a particle in $\bb T^d_N$ induced by the generator $L_N$ given as follows.
Let $\xi_{x,x+e_j}$ given by \eqref{rate}. If the particle is on a site $x\in \bb T^d_N$, it will jump to $x+e_j$ with rate $N^2\xi_{x,x+e_j}$. Furthermore, only nearest neighbor jumps are allowed.
The generator $\bb L_N$ of the random walk $\{X_t\}_{t\ge0}$ acts on functions $f:\bb T^d_N \to \bb R$ as
\begin{equation*}
\bb L_N f\left(\frac{x}{N}\right) \; =\; \sum^d_{j=1} \bb L_N^j f\left(\frac{x}{N}\right),
\end{equation*}
where,
\begin{equation*}
\bb L_N^j f\Big(\frac{x}{N}\Big) = N^2\Big\{\xi_{x,x+e_j} \Big[f\Big(\frac{x+e_j}{N}\Big) - f\Big(\frac{x}{N}\Big)\Big]  +
\xi_{x-e_j,x} \Big[f\Big(\frac{x-e_j}{N}\Big) - f\Big(\frac{x}{N}\Big)\Big]\Big\}
\end{equation*}
It is not difficult to see that the following equality holds: 
\begin{equation}
\label{opdisc}
\bb L_N f(x/N) = \sum^d_{j=1}\partial^N_{x_j}(a^N_j\partial^N_{W_j}f)(x)\;:=\;\nabla^NA^N\nabla^N_Wf(x). 
\end{equation}

Note that several properties of the above operator have been obtained in Section \ref{sec5}.  
 The counting measure $m_N$ on $N^{-1} \bb T^d_N$ is
reversible for this process.
This random walk plays an important role in the proof of the hydrodynamic limit of the process $\eta_t$, as we will see in subsection 7.3.

Let $D(\bb R_+, \{0,1\}^{\bb T^d_N})$ be the path space of
c\`adl\`ag trajectories with values in $\{0,1\}^{\bb T^d_N}$. For a
measure $\mu_N$ on $\{0,1\}^{\bb T^d_N}$, denote by $\bb P_{\mu_N}$ the
probability measure on $D(\bb R_+, \{0,1\}^{\bb T^d_N})$ induced by the
initial state $\mu_N$ and the Markov process $\{\eta_t : t\ge 0\}$.
Expectation with respect to $\bb P_{\mu_N}$ is denoted by $\bb
E_{\mu_N}$.

\begin{remark}\label{taxac}
The specific form of the rates $c_{x,x+e_i}$ is not important, but two
conditions must be fulfilled. The rates must be strictly positive,
they may not depend on the occupation variables $\eta(x)$, $\eta(x+e_i)$,
but they have to be chosen in such a way that the resulting process is
\emph{gradient}. (cf. Chapter 7 in \cite{kl} for the definition of
gradient processes).

We may define rates $c_{x,x+e_i}$ to obtain any polynomial $\Phi$ of the
form $\Phi(\alpha) = \alpha + \sum_{2\le j\le m} a_j \alpha^j$, $m\ge
1$, with $1+ \sum_{2\le j\le m} j a_j >0$. Let, for instance, $m=3$.
Then the rates
\begin{align*}
\hat c_{x,x+e_i} (\eta)\;\;& =\;\; c_{x,x+e_i} (\eta)\;\;+\\
&b\left\{ \eta(x-2e_i) \eta(x-e_i) + \eta(x-e_i) \eta(x+2e_i) + \eta(x+2e_i)
\eta(x+3e_i)\right\},
\end{align*}
satisfy the above three conditions, 
where $c_{x,x+e_i}$ is the rate defined at the beginning of Section 2 and
$a$, $b$ are such that $1+2a + 3b>0$. An elementary computation shows
that  
$\Phi(\alpha) = 1 + a \alpha^2 + b \alpha^3$.
\end{remark}

\subsection{The hydrodynamic equation}
\label{ss2.3}
The hydrodynamic equation is, roughly, a PDE that describes the time evolution of the thermodynamical quantities of the model in a fluid.

Let $A =(a_{jj})_{d\times d}$ be a diagonal matrix with $a_{jj}>0, j=1,\ldots,d$, and consider the operator
\begin{equation*}
\nabla A\nabla_W := \sum_{j=1}^d a_{jj}\partial_{x_j} \partial_{W_j} 
\end{equation*}
defined on $\bb D_W$.

A sequence of probability measures $\{\mu_N : N\geq 1 \}$ on $\{0,1\}^{\bb T^d_N}$
is said to be associated to a profile $\rho_0 :\bb T^d \to [0,1]$ if
\begin{equation}
\label{f09}
\lim_{N\to\infty}
\mu_N \left\{ \, \Big\vert \frac 1{N^d} \sum_{x\in\bb T^d_N} H(x/N) \eta(x)
- \int H(u) \rho_0(u) du \Big\vert > \delta \right\} \;=\; 0
\end{equation}
for every $\delta>0$ and every function $H\in \bb D_W$. 

Let $\gamma : \bb T^d \to [l,r]$ be a bounded density profile and consider the parabolic differential equation
\begin{equation}
\label{g03}
\left\{
\begin{array}{l}
{\displaystyle \partial_t \rho \; =\; \nabla A\nabla_W \Phi(\rho) } \\
{\displaystyle \rho(0,\cdot) \;=\; \gamma(\cdot)}
\end{array}
\right. ,
\end{equation}
where the function $\Phi:[l,r]\to\bb R$ is given as in the beginning of Section \ref{sec4}, and $t\in[0,T]$, for $T>0$ fixed.

Recall, from Section \ref{sec4}, that a bounded function $\rho : \bb [0,T] \times \bb T^d \to [l,r]$
is said to be a weak solution of the parabolic differential equation \eqref{g03} if the following conditions hold.  $\Phi(\rho(\cdot,\cdot))$ and $\rho(\cdot,\cdot)$ belong to $L^2([0,T],H_{1,W}(\bb T^d))$, and we have the integral identity
\begin{eqnarray*}
\int_{\bb T^d} \rho(t,u) H(u)du - \int_{\bb T^d} \rho(0,u) H(u)du=
\int_0^t \, \int_{\bb T^d} \Phi (\rho(s,u))  \nabla A\nabla_W H(u)du\,ds  \;,\end{eqnarray*}
for every function $H\in \bb D_W$ and all $t\in[0,T]$. 

Existence of such weak solutions follow from the tightness of the process proved in subsection 7.3, and from the energy estimate obtained in Lemma \ref{s03}. Uniquenesses of weak solutions was proved in subsection \ref{unicidade}.

\begin{theorem}
\label{t02}
Fix a continuous initial profile $\rho_0 : \bb T^d \to [0,1]$ and
consider a sequence of probability measures $\mu_N$ on $\{0,1\}^{\bb
  T^d_N}$ associated to $\rho_0$, in the sense of \eqref{f09}. Then, for any $t\ge 0$,
\begin{equation*}
\lim_{N\to\infty}
\bb P_{\mu_N} \left\{ \, \Big\vert \frac 1{N^d} \sum_{x\in\bb T^d_N}
H(x/N) \eta_t(x) - \int H(u) \rho(t,u)\, du \Big\vert
> \delta \right\} \;=\; 0
\end{equation*}
for every $\delta>0$ and every function $H\in \bb D_W$. Here, $\rho$
is the unique weak solution of the non-linear equation \eqref{g03}
with $l=0$, $r=1$, $\gamma = \rho_0$ and $\Phi(\alpha) = \alpha + a
\alpha^2$.
\end{theorem}

Let $\mc M$ be the space of positive measures on $\bb T^d$ with total
mass bounded by one endowed with the weak topology. Recall that
$\pi^{N}_{t} \in \mc M$ stands for the empirical measure at time $t$.
This is the measure on $\bb T^d$ obtained by rescaling space by $N$ and
by assigning mass $1/N^d$ to each particle:
\begin{equation}
\label{f01}
\pi^{N}_{t} \;=\; \frac{1}{N^d} \sum _{x\in \bb T^d_N} \eta_t (x)\,
\delta_{x/N}\;,
\end{equation}
where $\delta_u$ is the Dirac measure concentrated on $u$. 

For a function $H:\bb T^d \to \bb R$, $\<\pi^N_t, H\>$ stands for
the integral of $H$ with respect to $\pi^N_t$:
\begin{equation*}
\<\pi^N_t, H\> \;=\; \frac 1{N^d} \sum_{x\in\bb T^d_N}
H (x/N) \eta_t(x)\;.
\end{equation*}
This notation is not to be mistaken with the inner product in
$L^2(\bb T^d)$ introduced earlier. Also, when $\pi_t$ has a density
$\rho$, $\pi(t,du) = \rho(t,u) du$.

Fix $T>0$ and let $D([0,T], \mc M)$ be the space of $\mc M$-valued
c\`adl\`ag trajectories $\pi:[0,T]\to\mc M$ endowed with the
\emph{uniform} topology.  For each probability measure $\mu_N$ on
$\{0,1\}^{\bb T^d_N}$, denote by $\bb Q_{\mu_N}^{W,N}$ the measure on
the path space $D([0,T], \mc M)$ induced by the measure $\mu_N$ and
the process $\pi^N_t$ introduced in \eqref{f01}.

Fix a continuous profile $\rho_0 : \bb T^d \to [0,1]$ and consider a
sequence $\{\mu_N : N\ge 1\}$ of measures on $\{0,1\}^{\bb T^d_N}$
associated to $\rho_0$ in the sense \eqref{f09}. Further, we denote by $\bb Q_{W}$
the probability measure on $D([0,T], \mc M)$ concentrated on the
deterministic path $\pi(t,du) = \rho (t,u)du$, where $\rho$ is the
unique weak solution of \eqref{g03} with $\gamma = \rho_0$, $l_k=0$,
$r_k=1$, $k=1,\ldots,d$ and $\Phi(\alpha) = \alpha + b\alpha^2$.

In subsection \ref{ss1} we show that the sequence $\{\bb Q_{\mu_N}^{W,N} : N\ge
1\}$ is tight, and in subsection \ref{ss2} we characterize the limit
points of this sequence.

\subsection{Tightness}
\label{ss1}
The goal of this subsection is to prove tightness of sequence $\{\bb Q_{\mu_N}^{W,N} : N\ge 1\}$.
We will do it by showing that the set of equicontinuous paths of the empirical measures \eqref{f01} has probability close to one. 

Fix $\lambda >0$ and consider, initially, the auxiliary $\mc
M$-valued Markov process $\{\Pi^{\lambda,N}_t : t\ge 0\}$
defined by
\begin{equation*}
\Pi^{\lambda,N}_t (H) \;=\; \< \pi^N_t,H_\lambda^N\>\;=\;
\frac{1}{N^d} \sum _{x\in \bb Z^d} H_\lambda^N(x/N)
\eta_t (x),
\end{equation*}
for $H$ in $\bb D_W$, where $H_\lambda^N$ is the unique weak solution in $H_{1,W}(\bb T_N^d)$ (see Section \ref{sec5}) of
$$\lambda H_\lambda^N - \nabla^N A^N\nabla_W^N H_\lambda^N = \lambda H - \nabla A \nabla_W H,$$
with the right-hand side being understood as the restriction of the function to the lattice $\bb T_N^d$ (see subsection \ref{correspondencia}).

We first prove tightness of the process $\{\Pi^{\lambda,N}_t : 0\le t
\le T\}$,then we show that $\{\Pi^{\lambda,N}_t
: 0\le t \le T\}$ and $\{\pi^{N}_t : 0\le t \le T\}$ are not far apart.

It is well known \cite{kl} that to prove
tightness of $\{\Pi^{\lambda,N}_t : 0\le t \le T\}$ it is enough to
show tightness of the real-valued processes $\{\Pi^{\lambda,N}_t (H) :
0\le t \le T\}$ for a set of smooth functions $H:\bb T^d\to \bb R$ dense
in $C(\bb T^d)$ for the uniform topology.

Fix a smooth function $H: \bb T^d \to \bb R$. Keep in mind that $\Pi^{\lambda,N}_t (H) = \<\pi^N_t, H_\lambda^N \>$,
and denote by $M^{N,\lambda}_t$ the martingale defined by
\begin{equation}
\label{f10}
M^{N,\lambda}_t \;=\;  \Pi^{\lambda,N}_t (H) \;-\;
\Pi^{\lambda,N}_0 (H) \;-\; \int_0^t ds \, N^2 L_N \<\pi^N_s ,
H_\lambda^N \> \;.
\end{equation}
Clearly, tightness of $\Pi^{\lambda,N}_t (H)$ follows from tightness
of the martingale $M^{N,\lambda}_t$ and tightness of the additive
functional $\int_0^t ds \, N^2 L_N \<\pi^N_s , H_\lambda^N \>$.

A long computation, albeit simple, shows that the quadratic variation
$\<M^{N,\lambda}\>_t$ of the martingale $M^{N,\lambda}_t$ is given by:
\begin{align*}
 \frac{1}{N^{2d-1}}\sum^d_{j=1}\sum_{x\in \bb T^d}[\partial_{W,j}^N H^N_{\lambda}(x/N)]^2&[W((x+e_j)/N) - W(x/N)]\times\\
&\times\int_0^t c_{x,x+e_j}(\eta_s) \, [\eta_s(x+e_j) - \eta_s(x)]^2 \, ds\;.
\end{align*}
 In particular, by Lemma \ref{lmdiscreto},
\begin{equation*}
  \<M^{N,\lambda}\>_t \;\le\; \frac{C_0 t}{N^{2d-1}} \sum^d_{j=1} \|H_\lambda^N\|_{W_j,N}^2 \;\le\; \frac{C(H)t}{\lambda N^d},
\end{equation*}
for some finite constant $C(H)$, which depends only on $H$. Thus, by
Doob inequality, for every $\lambda>0$, $\delta>0$,
\begin{equation}
\label{f02}
\lim_{N\to\infty} \bb P_{\mu_N} \left[ \sup_{0\le t\le T}
\big\vert M^{N,\lambda}_t \big\vert \, > \, \delta \right]
\;=\; 0\;.
\end{equation}
In particular, the sequence of martingales $\{M^{N,\lambda}_t : N\ge
1\}$ is tight for the uniform topology.

It remains to be examined the additive functional of the decomposition
\eqref{f10}. The  generator of the exclusion process $L_N$  can be decomposed
 in terms of the generator of the random walk $\bb L_{N}$.  By a long but simple computation, we obtain that
 $N^2 L_N \<\pi^N , H_\lambda^N \>$ is equal to
\begin{eqnarray*}
\!\!\!\!\!\!\!\!\!\!\!\!\!\! &&
\sum^d_{j=1}\big \{\frac {1}{N^d} \sum_{x\in \bb T^d_N} (\bb L_N^j H_\lambda^N)(x/N)\, \eta(x)
\\
\!\!\!\!\!\!\!\!\!\!\!\!\!\! && \quad
+\; \frac{b}{N^d} \sum_{x\in \bb T^d_N} \big [ (\bb L_N^j H_\lambda^N)
((x+e_j)/N) + (\bb L_{N}^j H_\lambda^N) (x/N) \big ] \,
(\tau_x h_{1,j}) (\eta) \\
\!\!\!\!\!\!\!\!\!\!\!\!\!\! && \qquad
- \; \frac{b}{N^d} \sum_{x\in \bb T^d_N} (\bb L_N^j H_\lambda^N)
(x/N) (\tau_x h_{2,j}) (\eta)\big \}\;,
\end{eqnarray*}
where $\{\tau_x: x\in \bb Z^d\}$ is the group of translations, so that
$(\tau_x \eta)(y) = \eta(x+y)$ for $x$, $y$ in $\bb Z^d$, and the
sum is understood modulo $N$. Also, $h_{1,j}$, $h_{2,j}$ are the cylinder functions
\begin{equation*}
h_{1,j}(\eta) \;=\; \eta(0) \eta({e_j})\;,\quad h_{2,j}(\eta) \;=\; \eta(-e_j)  \eta(e_j)\;.
\end{equation*}

Since $H_\lambda^N$ is the weak solution of the discrete equation, we have by Remark \ref{diracdiscreta} that it is also a strong solution. Then, we may replace $\bb L_N H_\lambda^N$ by $U_\lambda^N
= \lambda H_\lambda^N - H$ in the previous formula. In particular,
for all $0\le s<t\le T$,
\begin{equation*}
\Big\vert \int_s^t dr \, N^2 L_N \<\pi^N_r ,H_\lambda^N \> \Big\vert
\;\le\; \frac {(1+3|b|)(t-s)}{N^d} \sum_{x\in \bb T^d_N} |U_\lambda^N (x/N)|
\;.
\end{equation*}
It follows from the estimate given in Lemma \ref{lmdiscreto}, and from Schwartz
inequality, that the right hand side of the previous expression is bounded above by $C(H,b) (t-s)$
uniformly in $N$, where $C(H,b)$ is a finite constant depending only
on $b$ and $H$. This proves that the additive part of the
decomposition \eqref{f10} is tight for the uniform topology and,
therefore, that the sequence of processes $\{\Pi^{\lambda,N}_t :N\ge
1\}$ is tight.

\begin{lemma}
\label{s06}
The sequence of measures $\{\bb Q_{\mu^N}^{W,N} : N\ge 1\}$ is tight
for the uniform topology.
\end{lemma}

\begin{proof}
Fix $\lambda > 0$. It is enough to show that for every function $H\in \bb D_W$
and every $\epsilon>0$, we have
\begin{equation*}
\lim_{N\to\infty} \bb P_{\mu^N} \left[
\sup_{0\le t\le T} |\, \Pi^{\lambda,N}_t (H) -
\<\pi^N_t, H\>\, | > \epsilon
\right] \;=\;0,
\end{equation*}
whence tightness of $\pi^N_t$ follows from 
tightness of $\Pi^{\lambda,N}_t$. By Chebyshev's inequality, the last expression is bounded above by
$$\bb E_{\mu_N} \left[\sup_{0\le t\le T} |\, \Pi^{\lambda,N}_t (H) -
\<\pi^N_t, H\>\, |^2\right] \leq 2\|H_\lambda^N - H\|_N^2,
$$
since there exists at most one particle per site. By Theorem \ref{homoge} and Proposition \ref{hconvergencia}, $\|H_\lambda^N - H\|_N^2\to 0$ as $N\to\infty$, and the proof follows.
\end{proof}

\subsection{Uniqueness of limit points}
\label{ss2}

We prove in this subsection that all limit points $\bb Q^*$ of the
sequence $\bb Q^{W,N}_{\mu_N}$ are concentrated on absolutely
continuous trajectories $\pi(t,du) = \rho(t,u) du$, whose density
$\rho(t,u)$ is a weak solution of the hydrodynamic equation
\eqref{g03} with $l=0$, $r=1$ and $\Phi(\alpha)=\alpha + a\alpha^2$.

We now state a result necessary to prove the uniqueness of limit points. Let, for a local function $g: \{0,1\}^{\bb Z^d} \to \bb R$, 
$\tilde g :[0,1]\to \bb R$ be the expected value of $g$ under the
stationary states:
\begin{equation*}
\tilde g (\alpha) \;=\; E_{\nu_\alpha} [ g(\eta)]\;.
\end{equation*}
For $\ell \ge 1$ and $d$-dimensional integer $x=(x_1,\ldots,x_d)$, denote 
by $\eta^{\ell} (x)$ the empirical density of particles in the box 
$ \bb B_+^\ell(x)= \{(y_1,\ldots,y_d)\in\bb Z^d\;;0\le y_i-x_i < \ell\}$: 
\begin{equation*}
  \eta^{\ell} (x) \;=\; \frac{1}{\ell^d}  \sum_{y\in \bb B_+^\ell(x)} \eta(y)\;.
\end{equation*}

\begin{proposition}[Replacement lemma]
\label{s02}
Fix a cylinder function $g$ and a sequence of functions $\{F_N : N\ge
1\}$, $F_N : N^{-1} \bb T^d_N\to \bb R$ such that
\begin{equation*}
\limsup_{N\to\infty} \frac 1{N^d} \sum_{x\in \bb T^d_N} F_N(x/N)^2 \;<\;
\infty\; .
\end{equation*}
Then, for any $t>0$ and any sequence of probability measures $\{\mu_N
: N\ge 1\}$ on $\{0,1\}^{\bb T^d_N}$,
\begin{equation*}
\limsup_{\varepsilon\to 0} \limsup_{N\to\infty}
\bb E_{\mu_N} \Big[ \, \Big| \int_0^t  \frac 1{N^d}
\sum_{x\in \bb T_N^d} F_N(x/N) \, \big \{ \tau_x g (\eta_s) - \tilde g
(\eta^{\varepsilon N}_s(x))\ d_s \big \} \Big| \, \Big]  \;=\; 0\;.
\end{equation*}
\end{proposition}

The proof can be found in \cite[subsection 5.3]{v}.

Let $\bb Q^*$ be a limit point of the sequence $\bb Q^{W,N}_{\mu_N}$
and assume, without loss of generality, that $\bb Q^{W,N}_{\mu_N}$
converges to $\bb Q^*$.

Since there is at most one particle per site, it is clear that $\bb
Q^*$ is concentrated on trajectories $\pi_t(du)$ which are absolutely
continuous with respect to the Lebesgue measure, $\pi_t(du) =
\rho(t,u) du$, and whose density $\rho$ is non-negative and bounded by
$1$.

Fix a function $H\in\bb D_W$  and
$\lambda>0$.  Recall the definition of the martingale
$M^{N,\lambda}_t$ introduced in the previous section. From \eqref{f02} we have,
for every $\delta>0$,
\begin{equation*}
\lim_{N\to\infty} \bb P_{\mu_N} \left[ \sup_{0\le t\le T}
\big\vert M^{N,\lambda}_t \big\vert \, > \, \delta \right]
\;=\; 0\;,
\end{equation*}
and from \eqref{f10}, for fixed $0<t\le T$ and $\delta>0$, we have
\begin{equation*}
\lim_{N\to\infty} \bb Q^{W,N}_{\mu_N} \left[ \,
\Big\vert \<\pi^N_t, H^N_\lambda \> \;-\;
\<\pi^N_0, H^N_\lambda \> \;-\;
\int_0^t ds \, N^2 L_N \<\pi^N_s , H^N_\lambda \>
\Big\vert \, > \, \delta \right] \;=\; 0.
\end{equation*}

Note that the expression $N^2 L_N \<\pi^N_s , H^N_\lambda\>$ has been computed in the previous subsection in terms of generator $\bb L_N$. On the other hand, $\bb L_N H_\lambda^N = \lambda H_\lambda^N - \lambda H + \nabla A\nabla_W H$. Since there is at most one particle per site, we may apply Theorem \ref{homoge} to replace $\<\pi^N_t,
H^N_\lambda \>$ and $\<\pi^N_0, H^N_\lambda \>$ by $\<\pi_t, H\>$ and $\<\pi_0,H\>$, respectively, and replace $\bb L_N H_\lambda^N$ by $\nabla A\nabla_W H$ plus a term that vanishes as $N\to\infty$.

Since $E_{\nu_\alpha}[h_{i,j}] = \alpha^2$, $i=1$, $2$ and $j = 1,\ldots, d$, we have by Proposition \ref{s02} that, for every $t>0$, $\lambda>0$,
$\delta>0$, $i=1$, $2$,
\begin{align*}
\lim_{\varepsilon \to 0} \limsup_{N\to\infty}
\bb P_{\mu_N} \Big[ \, \Big| \int_0^t \!\!\! ds\, \frac 1{N^d}
\sum_{x\in \bb T^d_N}& \nabla A\nabla_W H  (x/N) \times\\
&\times\left\{ \tau_x h_{i,j} (\eta_s) -
\left[\eta^{\varepsilon N}_s(x)\right]^2 \right\} \, \Big|
\, > \, \delta \, \Big]  \;=\; 0.
\end{align*}

Since $\eta^{\varepsilon N}_s(x) = \varepsilon^{-d} \pi^N_s (\prod_{j=1}^d[x_j/N, x_j/N + \varepsilon e_j])$,
 we obtain, from the previous considerations, that
\begin{align*}
\lim_{\varepsilon \to 0} \limsup_{N\to\infty} \bb Q^{W,N}_{\mu_N} \left[ \,
\Big\vert\right. &\<\pi_t, H \> \;-\; \\
-\; \<\pi_0, H \> \;-\;&\left.
\int_0^t ds \, \Big\< \Phi \big(\varepsilon^{-d} \pi^N_s (\prod_{j=1}^d[\cdot, \cdot
+ \varepsilon e_j]) \big) \,,\, \nabla A\nabla_W H\Big>
\Big\vert > \delta \right] \;=\; 0\;.
\end{align*}

Using the fact that $\bb Q^{W,N}_{\mu_N}$ converges in the uniform topology to $\bb
Q^*$, we have that
\begin{align*}
\lim_{\varepsilon \to 0}\bb Q*\left[ \,
\Big\vert \<\pi_t, G_\lambda H \> \right.&\;-\; \; \<\pi_0, G_\lambda H \> \;-\;\\
-\;\int_0^t ds \, &\left.\Big\< \Phi \big (\varepsilon^{-d} \pi_s (\prod_{j=1}^d[\cdot, \cdot
+ \varepsilon e_j]) \big) \,,\, U_\lambda\Big>
\Big\vert > \delta \right] \;=\; 0\;.
\end{align*}

Recall that $\bb Q^{*}$ is concentrated on absolutely continuous paths
$\pi_t(du) = \rho(t,u) du$ with positive density bounded by $1$. Therefore,
$\varepsilon^{-d}\pi_s(\prod_{j=1}^d[\cdot, \cdot + \varepsilon e_j])$ converges in
$L^1(\bb T^d)$ to $\rho(s,.)$ as $\varepsilon\downarrow 0$. Thus,
\begin{eqnarray*}
\bb Q^{*} \left[ \,
\Big\vert \<\pi_t, H \> \;-\;
 \<\pi_0, H \> \;-\;
\int_0^t ds \, \< \Phi (\rho_s) \,,\, \nabla A\nabla_W H \>
\Big\vert > \delta \right] \;=\; 0.
\end{eqnarray*}
Letting $\delta\downarrow
0$, we see that, $\bb Q^{*}$ a.s.,
\begin{eqnarray*}
\int_{\bb T^d} \rho(t,u) H(u)du - \int_{\bb T^d} \rho(0,u) H(u)du=
\int_0^t \, \int_{\bb T^d} \Phi (\rho(s,u))  \nabla A\nabla_W H(u)du\,ds  \;.
\end{eqnarray*}

This identity can be extended to a countable set of times $t$. Taking
this set to be dense we
obtain, by continuity of the trajectories $\pi_t$, that it holds for all $0\le t\le T$.

We now have a lemma regarding the energy of such limit points whose proof can be found in \cite[Section 6]{v}.

\begin{lemma}
\label{s03}
There exists a finite constant $K_1$, depending only on $b$, such that
\begin{align*}
E_{\bb Q^*_{W}} \left[ \sup_{H\in\bb D_W} \left\{ \int_0^T ds\, \int_{\bb T^d}
dx\right.\right. &\, (\partial_{x_j} \partial_{W_j} H) (s, x) \, \Phi(\rho(s,x)) \\ 
- \; K_1 \int_0^T ds\, &\left.\left. \int_{\bb T^d} [\partial_{W_j} H (s, x)]^2
\, d(x^j\otimes W_j) \right\} \right] \; \le \; K_0.
\end{align*}
\end{lemma}

From Lemma \ref{s03}, we may conclude that all limit points have, almost surely, finite energy, and therefore, by Lemma \ref{energia}, $\Phi(\rho(\cdot,\cdot))\in L^2([0,T],H_{1,W}(\bb T^d))$. Analogously, it is possible to show that 
 $\rho(\cdot,\cdot)$ has finite energy and hence it belongs to $L^2([0,T],H_{1,W}(\bb T^d))$.
\begin{proposition}
\label{s15}
As $N\uparrow\infty$, the sequence of probability measures $\bb
Q_{\mu_N}^{W,N}$ converges in the uniform topology to $\bb Q_{W}$.
\end{proposition}
\begin{proof}
In the previous subsection, we showed that the sequence of probability
measures $\bb Q^{W,N}_{\mu_N}$ is tight for the uniform topology. Moreover, we
just proved that all limit points of this sequence are concentrated on
weak solutions of the parabolic equation \eqref{g03}. The proposition now follows
 from the uniqueness proved in subsection \ref{unicidade}.
\end{proof}

\begin{proof}[Proof of Theorem \ref{t02}]
Since $\bb Q_{\mu_N}^{W,N}$ converges in the uniform topology to $\bb
Q_{W}$, a measure which is concentrated on a deterministic path, for
each $0\le t\le T$ and each continuous function $H:\bb T^d\to \bb R$,
$\<\pi^N_t, H\>$ converges in probability to $\int_{\bb T^d} du
\rho(t,u)H(u)$, where $\rho$ is the unique weak solution of
\eqref{g03} with $l_k=0$, $r_k=1$, $\gamma=\rho_0$ and $\Phi(\alpha) =
\alpha + a \alpha^2$.
\end{proof}

\end{document}